\documentclass[11pt,a4paper]{article}
\usepackage{amsmath}
\usepackage{amsfonts}
\usepackage{amsthm}
\usepackage[pdftex]{color,graphicx}
\usepackage{geometry}
\usepackage{setspace}
\usepackage{multirow}
\usepackage{algorithmic}
\usepackage{algorithm}
\usepackage{authblk}
\numberwithin{algorithm}{section}

\title{Structured Sparsity via Alternating Direction Methods\footnote{Research supported in part by DMS 10-16571, ONR Grant N00014-08-1-1118 and DOE Grant DE-FG02-08ER25856.}}
\author{Zhiwei (Tony) Qin \footnote{Email: zq2107@columbia.edu.}}
\author{Donald Goldfarb \footnote{Email: goldfarb@columbia.edu.}}
\affil{Department of Industrial Engineering and Operations Research \authorcr\itshape\small Columbia University \authorcr\itshape\small New York, NY 10027}

\begin{document}
\maketitle

\newtheorem{alg}{Algorithm}[section]
\newtheorem{thm}{Theorem}[section]
\newtheorem{cor}{Corollary}[section]
\newtheorem{lem}{Lemma}[section]
\newtheorem{rem}{Remark}[section]
\newcommand{\tildex}{\tilde{x}}
\newcommand{\tildey}{\tilde{y}}
\newcommand{\ybar}{\bar{y}}
\newcommand{\xbar}{\bar{x}}
\newcommand{\gradyfxybar}{\nabla_{\ybar}f(x,\ybar)}
\newcommand{\gradyfxy}{\nabla_{y}f(x,y)}
\newcommand{\gradxfxybar}{\nabla_xf(x,\ybar)}
\newcommand{\gradxfxy}{\nabla_xf(x,y)}
\newcommand{\glassonormy}{\lambda\sum_s \|y_s\|}
\newcommand{\mlassonormy}{\lambda\sum_s \|y_s\|_\infty}
\newcommand{\glassonormybar}{\lambda\sum_s \|\ybar_s\|}
\newcommand{\varpone}[2]{#1^{#2+1}}
\newcommand{\xnpone}{\varpone{x}{n}}
\newcommand{\ynpone}{\varpone{y}{n}}
\newcommand{\xkpone}{\varpone{x}{k}}
\newcommand{\ykpone}{\varpone{y}{k}}
\newcommand{\xKpone}{\varpone{x}{K}}
\newcommand{\yKpone}{\varpone{y}{K}}
\newcommand{\xlpone}{\varpone{x}{l}}
\newcommand{\ylpone}{\varpone{y}{l}}
\newcommand{\ybarnpone}{\ybar^{n+1}}
\newcommand{\xbarnpone}{\xbar^{n+1}}
\newcommand{\ybarkpone}{\ybar^{k+1}}
\newcommand{\xbarkpone}{\xbar^{k+1}}
\newcommand{\xbarKpone}{\varpone{\xbar}{K}}
\newcommand{\ybarKpone}{\varpone{\ybar}{K}}
\newcommand{\zkpone}{z^{k+1}}
\newcommand{\pfxybar}{p_f(x,\ybar)}
\newcommand{\pfxy}{\bar{q}}
\newcommand{\pgxbary}{p_g(\xbar,y)}
\newcommand{\pgybar}{(p,q)}
\newcommand{\fxybar}{f(x,\ybar)}
\newcommand{\fxy}{f(x,y)}
\newcommand{\gxbary}{g(\xbar,y)}
\newcommand{\gxbarybar}{g(\xbar,\ybar)}
\newcommand{\gybar}{g(\ybar)}
\newcommand{\gammagxbary}{\gamma_g(\xbar,y)}
\newcommand{\gammagxbarybar}{\gamma_g(\xbar,\ybar)}
\newcommand{\gammagybar}{\gamma_g(\ybar)}
\newcommand{\yprime}{y^\prime}
\newcommand{\xprime}{x^\prime}
\newcommand{\tkpone}{t_{k+1}}
\newcommand{\ygentry}[1]{y^{(#1)}_g}
\newcommand{\ygi}{\ygentry{i}}
\newcommand{\cgentry}[1]{c^{(#1)}_g}
\newcommand{\cgi}{\cgentry{i}}
\newcommand{\zgentry}[1]{z^{(#1)}_g}
\newcommand{\Lrho}{\mathcal{L}_\rho}

\newcommand{\Keywords}[1]{\par\noindent
{\small{\em Keywords\/}: #1}}

\begin{abstract}
We consider a class of sparse learning problems in high dimensional feature space regularized by a structured sparsity-inducing norm that incorporates prior knowledge of the group structure of the features.  Such problems often pose a considerable challenge to optimization algorithms due to the non-smoothness and non-separability of the regularization term.  In this paper, we focus on two commonly adopted sparsity-inducing regularization terms, the overlapping Group Lasso penalty $l_1/l_2$-norm and the $l_1/l_\infty$-norm.  We propose a unified framework based on the augmented Lagrangian method, under which problems with both types of regularization and their variants can be efficiently solved.  As one of the core building-blocks of this framework, we develop new algorithms using a partial-linearization/splitting technique and prove that the accelerated versions of these algorithms require $O(\frac{1}{\sqrt{\epsilon}})$ iterations to obtain an $\epsilon$-optimal solution.  We compare the performance of these algorithms against that of the alternating direction augmented Lagrangian and FISTA methods on a collection of data sets and apply them to two real-world problems to compare the relative merits of the two norms.
\\
\end{abstract}

\section{Introduction}\label{sec:intro}
For feature learning problems in a high-dimensional space, sparsity in the feature vector is usually a desirable property.  Many statistical models have been proposed in the literature to enforce sparsity, dating back to the classical Lasso model ($l_1$-regularization) \cite{tib, chen1999atomic}.  The Lasso model is particularly appealing because it can be solved by very efficient proximal gradient methods; e.g., see \cite{combettes2009proximal}. However, the Lasso does not take into account the structure of the features \cite{zou2005regularization}.  In many real applications, the features in a learning problem are often highly correlated, exhibiting a group structure.  Structured sparsity has been shown to be effective in those cases.  The Group Lasso model \cite{yuanlin,bach2008consistency,roth2008group} assumes disjoint groups and enforces sparsity on the pre-defined groups of features.  This model has been extended to allow for groups that are hierarchical as well as overlapping \cite{jenatton2009structured, kim2009tree, bach2010structured} with a wide array of applications from gene selection \cite{kim2009tree} to computer vision \cite{huang2009learning, jenatton2009spca}.  For image denoising problems, extensions with non-integer block sizes and adaptive partitions have been proposed in \cite{peyre2011group, peyre2011adaptive}.  In this paper, we consider the following basic model of minimizing the squared-error loss with a regularization term to induce group sparsity:
\begin{equation}\label{eq:overlap_glasso}
    \min_{x\in \mathbb{R}^m} L(x) + \Omega(x),
\end{equation}
where
\begin{eqnarray}
  \nonumber L(x) &=& \frac{1}{2}\|Ax-b\|^2, \quad\quad A \in \mathbb{R}^{n\times m},\\
   \Omega(x) &=& \left\{
                            \begin{array}{ll}
                              \Omega_{l_1/l_2}(x) \equiv \lambda\sum_{s\in \mathcal{S}}w_s\|x_s\|, & \hbox{or} \\
                              \Omega_{l_1/l_\infty}(x) \equiv \lambda\sum_{s\in \mathcal{S}}w_s\|x_s\|_\infty & \hbox{,}
                            \end{array}
                          \right. \label{eq:Omega_def}
\end{eqnarray}
$\mathcal{S}=\{s_1,\cdots,s_{|\mathcal{S}|}\}$ is the set of group indices with $|\mathcal{S}| = J$, and the elements (features) in the groups possibly overlap \cite{chen2010efficient, obozinski2010network, jenatton2009structured, bach2010structured}.  In this model, $\lambda, w_s, \mathcal{S}$ are all pre-defined.  $\|\cdot\|$ without a subscript denotes the $l_2$-norm.  We note that the penalty term $\Omega_{l_1/l_2}(x)$ in \eqref{eq:Omega_def} is different from the one proposed in \cite{jacob2009group}, although both are called overlapping Group Lasso penalties.  In particular, \eqref{eq:overlap_glasso}-\eqref{eq:Omega_def} cannot be cast into a non-overlapping group lasso problem as done in \cite{jacob2009group}.  

\subsection{Related Work}
Two proximal gradient methods have been proposed to solve a close variant of \eqref{eq:overlap_glasso} with an $l_1/l_2$ penalty,
\begin{equation}\label{eq:sparse_ogl}
    \min_{x\in \mathbb{R}^m} L(x) + \Omega_{l_1/l_2}(x) + \lambda\|x\|_1,
\end{equation} 
which has an additional $l_1$-regularization term on $x$.  Chen et al. \cite{chen2010efficient} replace $\Omega_{l_1/l_2}(x)$ with a smooth approximation $\Omega_\eta(x)$ by using Nesterov's smoothing technique \cite{nesterov2005smooth} and solve the resulting problem by the Fast Iterative Shrinkage Thresholding algorithm (FISTA) \cite{beckteboulle}.  The parameter $\eta$ is a smoothing parameter, upon which the practical and theoretical convergence speed of the algorithm critically depends.  Liu and Ye \cite{liu2010fast} also apply FISTA to solve \eqref{eq:sparse_ogl}, but in each iteration, they transform the computation of the proximal operator associated with the combined penalty term into an equivalent constrained smooth problem and solve it by Nesterov's accelerated gradient descent method \cite{nesterov2005smooth}.  Mairal et al. \cite{obozinski2010network} apply the accelerated proximal gradient method to \eqref{eq:overlap_glasso} with $l_1/l_\infty$ penalty and propose a network flow algorithm to solve the proximal problem associated with $\Omega_{l_1/l_\infty}(x)$.  Mosci et al.'s method \cite{mosci2010primal} for solving the Group Lasso problem in \cite{jacob2009group} is in the same spirit as \cite{liu2010fast}, but their approach uses a projected Newton method.

\subsection{Our Contributions}
We take a unified approach to tackle problem \eqref{eq:overlap_glasso} with both $l_1/l_2$- and $l_1/l_\infty$-regularizations.  Our strategy is to develop efficient algorithms based on the Alternating Linearization Method with Skipping (ALM-S) \cite{goldfarb2009falm} and FISTA for solving an equivalent constrained version of problem \eqref{eq:overlap_glasso} (to be introduced in Section \ref{sec:reformulation}) in an augmented Lagrangian method framework.  Specifically, we make the following contributions in this paper:
\begin{itemize}
  \item We build a general framework based on the augmented Lagrangian method, under which learning problems with both $l_1/l_2$ and $l_1/l_\infty$ regularizations (and their variants) can be solved.  This framework allows for experimentation with its key building blocks.
  \item We propose new algorithms: ALM-S with partial splitting (APLM-S) and FISTA with partial linearization (FISTA-p), to serve as the key building block for this framework.  We prove that APLM-S and FISTA-p have convergence rates of $O(\frac{1}{k})$ and $O(\frac{1}{k^2})$ respectively, where $k$ is the number of iterations.  Our algorithms are easy to implement and tune, and they do not require line-search, eliminating the need to evaluate the objective function at every iteration.
  \item We evaluate the quality and speed of the proposed algorithms and framework against state-of-the-art approaches on a rich set of synthetic test data and compare the $l_1/l_2$ and $l_1/l_\infty$ models on breast cancer gene expression data \cite{van2002gene} and a video sequence background subtraction task \cite{obozinski2010network}.
\end{itemize}


\section{A Variable-Splitting Augmented Lagrangian Framework}\label{sec:reformulation}
In this section, we present a unified framework, based on variable splitting and the augmented Lagrangian method for solving \eqref{eq:overlap_glasso} with both $l_1/l_2$- and $l_1/l_\infty$-regularizations.  This framework reformulates problem \eqref{eq:overlap_glasso} as an equivalent linearly-constrained problem, by using the following variable-splitting procedure.


Let $y \in \mathbb{R}^{\sum_{s\in\mathcal{S}}|s|}$ be the vector obtained from the vector $x \in \mathbb{R}^m$ by repeating components of $x$ so that no component of $y$ belongs to more than one group.  Let $M = \sum_{s\in\mathcal{S}}|s|$.  The relationship between $x$ and $y$ is specified by the linear constraint $Cx = y$, where the $(i,j)$-th element of the matrix $C \in \mathbb{R}^{M\times m}$ is
\begin{equation}\label{eq:mat_C}
    C_{i,j} = \left\{
                \begin{array}{ll}
                  1, & \hbox{if $y_i$ is a replicate of $x_j$,} \\
                  0, & \hbox{otherwise.}
                \end{array}
              \right.
\end{equation}
For examples of $C$, refer to \cite{chen2010efficient}.  Consequently, \eqref{eq:overlap_glasso} is equivalent to
\begin{eqnarray}\label{eq:overlap_glasso_aug1}
     \min && F_{obj}(x,y) \equiv \frac{1}{2}\|Ax-b\|^2 + \tilde{\Omega}(y) \\
     \nonumber s.t. && Cx = y,
\end{eqnarray}
where $\tilde{\Omega}(y)$ is the non-overlapping group-structured penalty term corresponding to $\Omega(y)$ defined in \eqref{eq:Omega_def}.
Note that $C$ is a highly sparse matrix, and $D = C^TC$ is a diagonal matrix with the diagonal entries equal to the number of times that each entry of $x$ is included in some group.  Problem \eqref{eq:overlap_glasso_aug1} now includes two sets of variables $x$ and $y$, where $x$ appears only in the loss term $L(x)$ and $y$ appears only in the penalty term $\tilde{\Omega}(y)$.

All the non-overlapping versions of $\Omega(\cdot)$, including the Lasso and Group Lasso, are special cases of $\Omega(\cdot)$, with $C = I$, i.e. $x =y$.  Hence, \eqref{eq:overlap_glasso_aug1} in this case is equivalent to applying variable-splitting on $x$.  Problems with a composite penalty term, such as the Elastic Net, $\lambda_1 \|x\|_1 + \lambda_2 \|x\|^2$, can also be reformulated in a similar way by merging the smooth part of the penalty term ($\lambda_2\|x\|^2$ in the case of the Elastic Net) with the loss function $L(x)$.

To solve \eqref{eq:overlap_glasso_aug1}, we apply the augmented Lagrangian method \cite{hestenes1969multiplier,powell1972nonlinear,nocedal,bertsekas1999nonlinear} to it.  This method, Algorithm \ref{alg:aug_lag_exact}, minimizes
the augmented Lagrangian
\begin{equation}\label{eq:aug_lag}
    \mathcal{L}(x,y,v) = \frac{1}{2}\|Ax-b\|^2 - v^T(Cx - y) + \frac{1}{2\mu}\|Cx-y\|^2 + \tilde{\Omega}(y)
\end{equation}
exactly for a given Lagrange multiplier $v$ in every iteration followed by an update to $v$.  The parameter $\mu$ in \eqref{eq:aug_lag} controls the amount of weight that is placed on violations of the constraint $Cx = y$.
Algorithm \ref{alg:aug_lag_exact} can also be viewed as a dual ascent algorithm applied to $P(v) = \min_{x,y}\mathcal{L}(x,y,v)$ \cite{bertsekas1976multiplier}, where $v$ is the dual variable, $\frac{1}{\mu}$ is the step-length, and $Cx-y$ is the gradient $\nabla_v P(v)$.
\begin{algorithm}
\caption{AugLag}
\begin{algorithmic}[1]\label{alg:aug_lag_exact}
\STATE Choose $x^0,y^0,v^0$.
\FOR{$l = 0,1,\cdots$}
    \STATE $(\xlpone, \ylpone) \gets \arg\min_{x,y} \mathcal{L}(x,y,v^l)$ \label{line:solve_auglag_exact}
    \STATE $v^{l+1} \gets v^l - \frac{1}{\mu}(C\xlpone - \ylpone)$
    \STATE Update $\mu$ according to the chosen updating scheme.
\ENDFOR
\end{algorithmic}
\end{algorithm}
This algorithm does not require $\mu$ to be very small to guarantee convergence to the solution of problem \eqref{eq:overlap_glasso_aug1} \cite{nocedal}.  However, solving the problem in Line \ref{line:solve_auglag_exact} of Algorithm \ref{alg:aug_lag_exact} exactly can be very challenging in the case of structured sparsity.  We instead seek an approximate minimizer of the augmented Lagrangian via the abstract subroutine ApproxAugLagMin$(x,y,v)$. The following theorem \cite{rockafellar1973multiplier} guarantees the convergence of this inexact version of Algorithm \ref{alg:aug_lag_exact}.
  
\begin{thm}\label{thm:ial_conv}
Let $\alpha^l := \mathcal{L}(x^l,y^l,v^l) - \inf_{x\in\mathbb{R}^m,y\in\mathbb{R}^M}\mathcal{L}(x,y,v^l)$ and $F^*$ be the optimal value of problem \eqref{eq:overlap_glasso_aug1}.  Suppose problem \eqref{eq:overlap_glasso_aug1} satisfies the modified Slater's condition, and
\begin{equation}\label{eq:sub_acc_cond}
    \sum_{l=1}^\infty \sqrt{\alpha^l} < +\infty.
\end{equation}
Then, the sequence $\{v^l\}$ converges to $v^*$, which satisfies
\begin{equation*}
    \inf_{x\in\mathbb{R}^m,y\in\mathbb{R}^M}\big(F_{obj}(x,y) - (v^*)^T(Cx-y)\big) = F^*,
\end{equation*}
while the sequence $\{x^l,y^l\}$ satisfies $\lim_{l\rightarrow\infty} Cx^l-y^l = 0$ and $\lim_{l\rightarrow\infty} F_{obj}(x^l,y^l) = F^*$.
\end{thm}
The condition \eqref{eq:sub_acc_cond} requires the augmented Lagrangian subproblem be solved with increasing accuracy.  We formally state this framework in Algorithm \ref{alg:aug_lag}.
\begin{algorithm}
\caption{OGLasso-AugLag}
\begin{algorithmic}[1]\label{alg:aug_lag}
\STATE Choose $x^0,y^0,v^0$.
\FOR{$l = 0,1,\cdots$}
    \STATE $(\xlpone, \ylpone) \gets$ ApproxAugLagMin$(x^l, y^l, v^l )$, to compute an approximate minimizer of $\mathcal{L}(x,y,v^l)$ \label{line:solve_auglag}
    \STATE $v^{l+1} \gets v^l - \frac{1}{\mu}(C\xlpone - \ylpone)$
    \STATE Update $\mu$ according to the chosen updating scheme.
\ENDFOR
\end{algorithmic}
\end{algorithm}
We index the iterations of Algorithm \ref{alg:aug_lag} by $l$ and call them `outer iterations'.  In Sections \ref{sec:partial_linearization}, we develop algorithms that implement ApproxAugLagMin$(x,y,v)$.  The iterations of these subroutine are indexed by $k$ and are called `inner iterations'.

\section{Methods for Approximately Minimizing the Augmented Lagrangian}\label{sec:partial_linearization}
In this section, we use the overlapping Group Lasso penalty $\Omega(x) = \lambda\sum_{s\in \mathcal{S}}w_s\|x_s\|$ to illustrate the optimization algorithms under discussion.  The case of $l_1/l_\infty$-regularization will be discussed in Section \ref{sec:L1infty}.  From now on, we assume without loss of generality that $w_s = 1$ for every group $s$.

\subsection{Alternating Direction Augmented Lagrangian (ADAL) Method}\label{sec:ADAL}
The well-known Alternating Direction Augmented Lagrangian (ADAL) method \cite{eckstein1992douglas, gabay1976dual, glowinski1975adal, boyd2010distributed}\footnote{Recently, Mairal et al. \cite{mairal2011network} also applied ADAL with two variants based on variable-splitting to the overlapping Group Lasso problem.} approximately minimizes the augmented Lagrangian by minimizing \eqref{eq:aug_lag} with respect to $x$ and $y$ alternatingly and then updates the Lagrange multiplier $v$ on each iteration (e.g., see \cite{bertsekas1989parallel}, Section 3.4).  Specifically, the single-iteration procedure that serves as the procedure ApproxAugLagMin$(x,y,v)$ is given below as Algorithm \ref{alg:adal}.


\begin{algorithm}
\caption{ADAL}
\begin{algorithmic}[1]\label{alg:adal}
\STATE Given $x^l$, $y^l$, and $v^l$.
    \STATE $\xlpone \gets \arg\min_x \mathcal{L}(x,y^l,v^l)$ \label{line:adal_x}
    \STATE $\ylpone \gets \arg\min_y \mathcal{L}(\xlpone,y,v^l)$ \label{line:adal_y}
\RETURN $\xlpone, \ylpone$.
\end{algorithmic}
\end{algorithm}

The ADAL method, also known as the alternating direction method of multipliers (ADMM) and the split Bregman method, has recently been applied to problems in signal and image processing \cite{combettes2009proximal, afonso2009augmented, goldstein2009split} and low-rank matrix recovery \cite{lin2010augmented}.  Its convergence has been established in \cite{eckstein1992douglas}.  This method can accommodate a sum of more than two functions.  For example, by applying variable-splitting (e.g., see \cite{bertsekas1989parallel, boyd2010distributed}) to the problem
$\min_x f(x) + \sum_{i=1}^K g_i(C_i x)$,
it can be transformed into
\begin{eqnarray*}
  \min_{x,y_1,\cdots,y_K} && f(x) + \sum_{i=1}^K g_i(y_i) \\
  s.t. && y_i = C_i x, \quad i = 1,\cdots,K.
\end{eqnarray*}
The subproblems corresponding to $y_i$'s can thus be solved simultaneously by the ADAL method.  This so-called simultaneous direction method of multipliers (SDMM) \cite{setzer2010deblurring} is related to Spingarn's method of partial inverses \cite{spingarn1983partial} and has been shown to be a special instance of a more general parallel proximal algorithm with inertia parameters \cite{pesquet2010parallel}.

Note that the problem solved in Line \ref{line:adal_y} of Algorithm \ref{alg:adal},
\begin{equation}\label{eq:adal_y_prob}
    \ylpone = \arg\min_y\mathcal{L}(\xlpone,y,v^l) \equiv \arg\min_y\left\{ \frac{1}{2\mu}\|d^l - y\|^2 + \tilde{\Omega}(y) \right\},
\end{equation}
where $d^l = C\xlpone - \mu v^l$, is group-separable and hence can be solved in parallel.  As in \cite{qin2010efficient}, each subproblem can be solved by applying the block soft-thresholding operator, $T(d^l_s,\mu\lambda) \equiv \frac{d^l_s}{\|d^l_s\|}\max(0, \|d^l_s\| - \lambda\mu), s = 1,\cdots,J$.  Solving for $\xlpone$ in Line \ref{line:adal_x} of Algorithm \ref{alg:adal}, i.e.
\begin{equation}\label{eq:adal_x_prob}
    \xlpone = \arg\min_x\mathcal{L}(x,y^l,v^l) \equiv \arg\min_x\left\{\frac{1}{2}\|Ax-b\|^2 - (v^l)^TCx + \frac{1}{2\mu}\|Cx-y^l\|^2 \right\},
\end{equation}
involves solving the linear system
\begin{equation}\label{eq:Lx_lin_sys}
    (A^TA + \frac{1}{\mu}D)x = A^Tb + C^Tv^l + \frac{1}{\mu}C^Ty^l,
\end{equation}
where the matrix on the left hand side of \eqref{eq:Lx_lin_sys} has dimension $m\times m$.  Many real-world data sets, such as gene expression data, are highly under-determined.  Hence, the number of features ($m$) is much larger than the number of samples ($n$).  In such cases, one can use the Sherman-Morrison-Woodbury formula,
\begin{equation}\label{eq:sherman_morrison}
    (A^TA + \frac{1}{\mu}D)^{-1} = \mu D^{-1} - \mu^2D^{-1}A^T(I + \mu AD^{-1}A^T)^{-1}AD^{-1},
\end{equation}
and solve instead an $n\times n$ linear system involving the matrix $I + \mu AD^{-1}A^T$.  In addition, as long as $\mu$ stays the same, one has to factorize $A^TA + \frac{1}{\mu}D$ or $I + \mu AD^{-1}A^T$ only once and store their factors for subsequent iterations.

When both $n$ and $m$ are very large, it might be infeasible to compute or store $A^TA$, not to mention its eigen-decomposition, or the Cholesky decomposition of $A^TA + \frac{1}{\mu}D$.  In this case, one can solve the linear systems using the preconditioned Conjugate Gradient (PCG) method \cite{golub1996matrix}. Similar comments apply to the other algorithms proposed in Sections \ref{sec:alms_partial2} - \ref{sec:fistap} below.
Alternatively, we can apply FISTA to Line \ref{line:solve_auglag} in Algorithm \ref{alg:aug_lag} (see Section \ref{sec:fista_full}).

\subsection{ALM-S: partial split (APLM-S)}\label{sec:alms_partial2}
We now consider applying the Alternating Linearization Method with Skipping (ALM-S) from \cite{goldfarb2009falm} to approximately minimize \eqref{eq:aug_lag}.
In particular, we apply variable splitting (Section \ref{sec:reformulation}) to the variable $y$, to which the group-sparse regularizer $\tilde{\Omega}$ is applied, (the original ALM-S splits both variables $x$ and $y$,)  and re-formulate \eqref{eq:aug_lag} as follows.
\begin{eqnarray}
    \min_{x,y,\ybar} && \frac{1}{2}\|Ax-b\|^2 - v^T(Cx-y) + \frac{1}{2\mu}\|Cx-y\|^2 + \tilde{\Omega}(\ybar) \label{eq:aug_lag_split} \\
    \nonumber s.t. && y = \ybar.
\end{eqnarray}
Note that the Lagrange multiplier $v$ is fixed here.
Defining
\begin{eqnarray}
  f(x,y) &:=& \frac{1}{2}\|Ax-b\|^2 - v^T(Cx-y) + \frac{1}{2\mu}\|Cx-y\|^2, \label{eq:f_x_ybar}\\
  g(y) &=& \tilde{\Omega}(y) = \glassonormy, \label{eq:g_y}
\end{eqnarray}
problem \eqref{eq:aug_lag_split} is of the form
\begin{eqnarray}
  \min && f(x,y) + g(\ybar) \label{eq:aug_lag_split_simplified}\\
  \nonumber s.t. && y = \ybar,
\end{eqnarray}
to which we now apply partial-linearization.

\subsubsection{Partial linearization and convergence rate analysis}
Let us define
\begin{eqnarray}
  F(x,y) &:=& f(x,y) + g(y) = \mathcal{L}(x,y;v), \\
  \mathcal{L}_\rho (x,y,\ybar,\gamma) &:=& f(x,y) + \gybar + \gamma^T(\ybar-y) + \frac{1}{2\rho}\|\ybar-y\|^2, \label{eq:aug_lag_aug_lag_partial}
\end{eqnarray}
where $\gamma$ is the Lagrange multiplier in the augmented Lagrangian \eqref{eq:aug_lag_aug_lag_partial} corresponding to problem \eqref{eq:aug_lag_split_simplified}, and $\gamma_g(\ybar)$ is a sub-gradient of $g$ at $\ybar$.  We now present our partial-split alternating linearization algorithm to implement ApproxAugLagMin$(x,y,v)$ in Algorithm \ref{alg:aug_lag}.
\begin{algorithm}
\caption{APLM-S}
\begin{algorithmic}[1]\label{alg:alms_partial2}
\STATE Given $x^0, \ybar^0, v$.  Choose $\rho, \gamma^0$, such that $-\gamma^0 \in \partial g(\ybar^0)$.  Define $f(x,y)$ as in \eqref{eq:f_x_ybar}.
\FOR{$k = 0,1,\cdots$ until stopping criterion is satisfied}
    \STATE $(x^{k+1},y^{k+1}) \gets \arg\min_{x,y}\mathcal{L}_\rho(x,y,\ybar^k,\gamma^k)$. \label{line:almsp_xbarybar}
    \IF{$F(x^{k+1},y^{k+1}) > \mathcal{L}_\rho(x^{k+1},y^{k+1},\ybar^k,\gamma^k)$} \label{line:almsp_check}
        \STATE $y^{k+1} \gets \ybar^k$ \label{line:y_skip}
        \STATE $x^{k+1} \gets \arg\min_x f(x,\ykpone) \equiv\arg\min_x \mathcal{L}_\rho(x;y^{k+1},\ybar^k,\gamma^k)$ \label{line:x_skip}
    \ENDIF
    \STATE $\ybar^{k+1} \gets p_f(\xkpone,\ykpone) \equiv \arg\min_{\ybar} \Lrho(\xkpone,\ykpone,\ybar,\nabla_y f(\xkpone,\ykpone))$ \label{line:almsp_ybar}
    \STATE $\gamma^{k+1} \gets \nabla_{y}f(x^{k+1},y^{k+1}) - \frac{y^{k+1} - \ybar^{k+1}}{\rho}$
\ENDFOR
\RETURN $(\xKpone, \ybarKpone)$
\end{algorithmic}
\end{algorithm}

We note that in Line \ref{line:x_skip} in Algorithm \ref{alg:alms_partial2},
\begin{equation}\label{eq:x_skip_opt}
    x^{k+1} = \arg\min_x \mathcal{L}_\rho(x;y^{k+1},\ybar^k,\gamma^k) \equiv \arg\min_x f(x;y^{k+1}) \equiv \arg\min_x f(x;\ybar^k).
\end{equation}
Now, we have a variant of Lemma 2.2 in \cite{goldfarb2009falm}.
\begin{lem}\label{lem:lem22x}
For any $(x,y)$, if $\pfxy := \arg\min_{\ybar}\Lrho(x,y,\ybar,\nabla_y f(x,y))$ and
\begin{equation}\label{eq:cond1_lem22x}
    F(x,\pfxy) \leq \Lrho(x,y,\pfxy,\nabla_y f(x,y)),
\end{equation}
then for any $(\xbar,\ybar)$,
\begin{equation}\label{eq:lem22A}
    2\rho(F(\xbar,\ybar) - F(x,\pfxy)) \geq \|\pfxy - \ybar\|^2 - \|y - \ybar\|^2 + 2\rho((\xbar - x)^T\gradxfxy).
\end{equation}
Similarly, for any $\ybar$, if $\pgybar := \arg\min_{x,y}\Lrho(x,y,\ybar,-\gamma_g(\ybar))$ and
\begin{equation}\label{eq:cond1_lem22ybar}
    F(\pgybar) \leq \Lrho(\pgybar,\ybar,-\gamma_g(\ybar)),
\end{equation}
then for any $(x,y)$,
\begin{equation}\label{eq:lem22B}
    2\rho(F(x,y) - F(\pgybar)) \geq \|\pgybar_y - y \|^2 - \|\ybar - y\|^2.
\end{equation}
\end{lem}

\begin{proof}
See Appendix \ref{sec:pf_lemma}.
\end{proof}

Algorithm \ref{alg:alms_partial2} checks condition \eqref{eq:cond1_lem22ybar} at Line \ref{line:almsp_check} because the function $g$ is non-smooth and condition \eqref{eq:cond1_lem22ybar} may not hold no matter what the value of $\rho$ is.  When this condition is violated, a skipping step occurs in which the value of $y$ is set to the value of $\ybar$ in the previous iteration (Line \ref{line:y_skip}) and $\mathcal{L}_\rho$ re-minimized with respect to $x$ (Line \ref{line:x_skip}) to ensure convergence.  Let us define a \textit{regular iteration} of Algorithm \ref{alg:alms_partial2} to be an iteration where no skipping step occurs, i.e. Lines \ref{line:y_skip} and \ref{line:x_skip} are not executed.  Likewise, we define a \textit{skipping iteration} to be an iteration where a skipping step occurs.
Now, we are ready to state the iteration complexity result for APLM-S.

\begin{thm}\label{thm:alms_partial_complexity}
Assume that $\nabla_y f(x,y)$ is Lipschitz continuous with Lipschitz constant $L_y(f)$, i.e. for any $x$, $\|\nabla_y f(x,y) - \nabla_y f(x,z)\| \leq L_y(f)\|y-z\|$, for all $y$ and $z$.
For $\rho \leq \frac{1}{L_y(f)}$, the iterates $(x^k,\ybar^k)$ in Algorithm \ref{alg:alms_partial2} satisfy
\begin{equation}\label{eq:alms_partial_iter_complexity}
    F(x^k, \ybar^k) - F(x^*, y^*) \leq \frac{\|\ybar^0 - y^*\|^2}{2\rho(k+k_n)}, \quad \forall k,
\end{equation}
where $(x^*,y^*)$ is an optimal solution to \eqref{eq:aug_lag_split}, and $k_n$ is the number of regular iterations among the first $k$ iterations.
\end{thm}

\begin{proof}
See Appendix \ref{sec:pf_thm}.
\end{proof}
\begin{rem}\label{rem:rho_mu}
For Theorem \ref{thm:alms_partial_complexity} to hold, we need $\rho \leq \frac{1}{L_y(f)}$.
From the definition of $f(x,y)$ in \eqref{eq:f_x_ybar}, it is easy to see that $L_y(f) = \frac{1}{\mu}$ regardless of the loss function $L(x)$.  Hence, we set $\rho = \mu$, so that condition \eqref{eq:cond1_lem22x} in Lemma \ref{lem:lem22x} is satisfied.
\end{rem}

In Section \ref{sec:fista_partial}, we will discuss the case where the iterations entirely consist of skipping steps.  We will show that this is equivalent to ISTA \cite{beckteboulle} with partial linearization as well as a variant of ADAL.  In this case, the inner Lagrange multiplier $\gamma$ is redundant.

\subsubsection{Solving the subproblems}
We now show how to solve the subproblems in Algorithm \ref{alg:alms_partial2}.  First, observe that since $\rho = \mu$
\begin{eqnarray}
  \arg\min_{\ybar} \Lrho(x,y,\ybar,\nabla_y f(x,y)) &\equiv& \arg\min_{\ybar}\left\{ \gradyfxy^T\ybar + \frac{1}{2\mu}\|\ybar-y\|^2 + g(\ybar) \right\} \\
   &\equiv& \arg\min_{\ybar}\left\{ \frac{1}{2\mu}\|d - \ybar\|^2 + \glassonormybar \right\}, \label{eq:aplms_y_prob}
\end{eqnarray}
where $d = Cx - \mu v$.  Hence, $\ybar$ can be obtained by applying the block soft-thresholding operator $T(d_s,\mu\lambda)$ as in Section \ref{sec:ADAL}.  
Next consider the subproblem
\begin{equation}\label{eq:almsp_xy}
    \min_{(x,y)} \Lrho(x,y,\ybar,\gamma) \equiv \min_{(x,y)} \left\{ f(x,y) + \gamma^T(\ybar-y) + \frac{1}{2\mu}\|\ybar-y\|^2 \right\}.
\end{equation}
It is easy to verify that solving the linear system given by the optimality conditions for \eqref{eq:almsp_xy} by block Gaussian elimination yields the system
\begin{equation}\label{eq:alalms_linsys}
    \left( A^TA + \frac{1}{2\mu}D  \right)x = r_x + \frac{1}{2}C^Tr_y
\end{equation}
for computing $x$, where $r_x = A^Tb + C^Tv$ and $r_y = -v + \gamma + \frac{\ybar}{\rho}$.  Then $y$ can be computed as $(\frac{\mu}{2})(r_y + \frac{1}{\mu}Cx)$.

As in Section \ref{sec:ADAL}, only one Cholesky factorization of $A^TA + \frac{1}{2\mu}D$ is required for each invocation of Algorithm \ref{alg:alms_partial2}.  Hence, the amount of work involved in each iteration of Algorithm \ref{alg:alms_partial2} is comparable to that of an ADAL iteration.

It is straightforward to derive an accelerated version of Algorithm \ref{alg:alms_partial2}, which we shall refer to as FAPLM-S, that corresponds to a partial-split version of the FALM algorithm proposed in \cite{goldfarb2009falm} and also requires $O(\sqrt{\frac{L(f)}{\epsilon}})$ iterations to obtain an $\epsilon$-optimal solution.  In Section \ref{sec:fistap}, we present an algorithm FISTA-p, which is a special version of
FAPLM-S in which every iteration is a skipping iteration and which has a much simpler form than FAPLM-S, while having essentially the same iteration complexity.

It is also possible to apply ALM-S directly, which splits both $x$ and $y$, to solve the augmented Lagrangian subproblem.  Similar to \eqref{eq:aug_lag_split}, we reformulate \eqref{eq:aug_lag} as
\begin{eqnarray}
  \min_{(x,y),(\xbar,\ybar)} && \frac{1}{2}\|Ax-b\|^2 - v^T(Cx-y) + \frac{1}{2\mu}\|Cx-y\|^2 + \lambda\sum_s\|\ybar_s\| \label{eq:aug_lag_fullsplit} \\
  \nonumber s.t. && x = \xbar, \\
  \nonumber && y = \ybar.
\end{eqnarray}
The functions $f$ and $g$ are defined as in \eqref{eq:f_x_ybar} and \eqref{eq:g_y}, except that now we write $g$ as $\gxbarybar$ even though the variable $\xbar$ does not appear in the expression for $g$.  It can be shown that $\ybar$ admits exactly the same expression as in APLM-S, whereas $\xbar$ is obtained by a gradient step, $x - \rho\nabla_x f(x,y)$.  To obtain $x$, we solve the linear system
\begin{equation}\label{eq:full_x_sln}
    \left( A^TA + \frac{1}{\mu+\rho}D + \frac{1}{\rho}I \right)x = r_x + \frac{\rho}{\mu+\rho}C^Tr_y,
\end{equation}
after which $y$ is computed by $y = \left( \frac{\mu\rho}{\mu+\rho} \right)\left( r_y + \frac{1}{\mu}Cx  \right)$.

\begin{rem}\label{rem:rho_mu_alms_full}
For ALM-S, the Lipschitz constant for $\nabla f(x,y)$ $L_f = \lambda_{max}(A^TA) + \frac{1}{\mu}d_{max}$, where $d_{max} = \max_i D_{ii} \geq 1$.
For the complexity results in \cite{goldfarb2009falm} to hold, we need $\rho \leq \frac{1}{L_f}$.  Since $\lambda_{max}(A^TA)$ is usually not known, it is necessary to perform a backtracking line-search on $\rho$ to ensure that $F(x^{k+1},y^{k+1}) \leq \mathcal{L}_\rho(x^{k+1},y^{k+1},\xbar^k,\ybar^k,\gamma^k)$.  In practice, we adopted the following continuation scheme instead.  We initially set $\rho = \rho_0 = \frac{\mu}{d_{max}}$ and  decreased $\rho$ by a factor of $\beta$ after a given number of iterations until $\rho$ reached a user-supplied minimum value $\rho_{min}$.  This scheme prevents $\rho$ from being too small, and hence negatively impacting computational performance.  However, in both cases the left-hand-side of the system \eqref{eq:full_x_sln} has to be re-factorized every time $\rho$ is updated.
\end{rem}

As we have seen above, the Lipschitz constant resulting from splitting both $x$ and $y$ is potentially much larger than $\frac{1}{\mu}$.  Hence, partial-linearization reduces the Lipschitz constant and hence improves the bound on the right-hand-side of \eqref{eq:alms_partial_iter_complexity} and allows Algorithm \ref{alg:alms_partial2} to take larger step sizes (equal to $\mu$).  Compared to ALM-S, solving for $x$ in the skipping step (Line \ref{line:x_skip}) becomes harder.  Intuitively, APLM-S does a better job of `load-balancing' by managing a better trade-off between the hardness of the subproblems and the practical convergence rate.  

\subsection{ISTA: partial linearization (ISTA-p)}\label{sec:fista_partial}
We can also minimize the augmented Lagrangian \eqref{eq:aug_lag}, which we write as $\mathcal{L}(x,y,v) = f(x,y) + g(y)$ with $f(x,y)$ and $g(y)$ defined as in \eqref{eq:f_x_ybar} and \eqref{eq:g_y}, using a variant of ISTA that only linearizes $f(x,y)$ with respect to the $y$ variables.  As in Section \ref{sec:alms_partial2}, we can set $\rho = \mu$ and guarantee the convergence properties of ISTA-p (see Corollary \ref{col:aplms_complexity} below). Formally, let $(x,y)$ be the current iterate and $(x^+, y^+)$ be the next iterate. We compute $y^+$ by
\begin{eqnarray}
  y^+ &=& \arg\min_{\yprime} \Lrho(x,y,\yprime,\nabla_y f(x,y)) \\
   &=& \arg\min_{\yprime} \left\{ \frac{1}{2\mu}\sum_j(\|y^\prime_j - d_{y_j}\|^2 + \lambda\|y^\prime_j\|)  \right\} \label{eq:fista_y_prob},
\end{eqnarray}
where $d_y = Cx - \mu v$.  Hence the solution $y^+$ to problem \eqref{eq:fista_y_prob} is given blockwise by $T([d_y]_j, \mu\lambda), j = 1,\cdots,J$.

Now given $y^+$, we solve for $x^+$ by
\begin{eqnarray}
  \nonumber x^+ &=& \arg\min_{\xprime} f(\xprime,y^+) \\
   &=& \arg\min_{\xprime} \left\{ \frac{1}{2}\|A\xprime-b\|^2 - v^T(C\xprime - y^+) + \frac{1}{2\mu}\|C\xprime-y^+\|^2  \right\} \label{eq:fista_partial_x_prob}
\end{eqnarray}
The algorithm that implements subroutine ApproxAugLagMin$(x,y,v)$ in Algorithm \ref{alg:aug_lag} by ISTA with partial linearization is stated below as Algorithm \ref{alg:ista_partial}.

\begin{algorithm}
\caption{ISTA-p (partial linearization)}
\begin{algorithmic}[1]\label{alg:ista_partial}
\STATE Given $x^0, \ybar^0, v$.  Choose $\rho$.  Define $f(x,y)$ as in \eqref{eq:f_x_ybar}.
\FOR{$k = 0,1,\cdots$ until stopping criterion is satisfied}
    \STATE $x^{k+1} \gets \arg\min_x f(x;\ybar^k)$
    \STATE $\ybarkpone \gets \arg\min_y \Lrho(x^{k+1},\ybar^k,y,\nabla_y f(\xkpone,\ybar^k))$
\ENDFOR
\RETURN $(\xKpone,\ybarKpone)$
\end{algorithmic}
\end{algorithm}
As we remarked in Section \ref{sec:alms_partial2}, Algorithm \ref{alg:ista_partial} is equivalent to Algorithm \ref{alg:alms_partial2} (APLM-S) where every iteration is a skipping iteration.  Hence, we have from Theorem \ref{thm:alms_partial_complexity}.
\begin{cor}\label{col:aplms_complexity}
Assume $\nabla_y f(\cdot,\cdot)$ is Lipschitz continuous with Lipschitz constant $L_y(f)$.  For $\rho \leq \frac{1}{L_y(f)}$, the iterates $(x^k,\ybar^k)$ in Algorithm \ref{alg:ista_partial} satisfy
\begin{equation}\label{eq:ista_partial_iter_complexity}
    F(x^k, \ybar^k) - F(x^*, y^*) \leq \frac{\|\ybar^0 - y^*\|^2}{2\rho k}, \quad \forall k,
\end{equation}
where $(x^*,y^*)$ is an optimal solution to \eqref{eq:aug_lag_split}.
\end{cor}

It is easy to see that \eqref{eq:fista_y_prob} is equivalent to \eqref{eq:adal_y_prob}, and that \eqref{eq:fista_partial_x_prob} is the same as \eqref{eq:adal_x_prob} in ADAL.

\begin{rem}
We have shown that with a fixed $v$, the ISTA-p iterations are exactly the same as the ADAL iterations.  The difference between the two algorithms is that ADAL updates the (outer) Lagrange multiplier $v$ in each iteration, while in ISTA-p, $v$ stays the same throughout the inner iterations.  We can thus view ISTA-p as a variant of ADAL with delayed updating of the Lagrange multiplier.
\end{rem}

The `load-balancing' behavior discussed in Section \ref{sec:alms_partial2} is more obvious for ISTA-p.  As we will see in Section \ref{sec:fista_full}, if we apply ISTA (with full linearization) to minimize \eqref{eq:aug_lag}, solving for $x$ is simply a gradient step.  Here, we need to minimize $f(x,y)$ with respect to $x$ exactly, while being able to take larger step sizes in the other subproblem, due to the smaller associated Lipschitz constant.

\subsection{FISTA-p}\label{sec:fistap}
We now present an accelerated version FISTA-p of ISTA-p. FISTA-p is a special case of FAPLM-S with a skipping step occurring in every iteration.
We state the algorithm formally as Algorithm \ref{alg:fista_partial}.
\begin{algorithm}
\caption{FISTA-p (partial linearization)}
\begin{algorithmic}[1]\label{alg:fista_partial}
\STATE Given $x^0, \ybar^0, v$. Choose $\rho$, and $z^0 = \ybar^0$. Define $f(x,y)$ as in \eqref{eq:f_x_ybar}.
\FOR{$k = 0,1,\cdots,K$}
    \STATE $x^{k+1} \gets \arg\min_x f(x;z^k)$ \label{line:fistap_x}
    \STATE $\ybarkpone \gets \arg\min_y \Lrho(x^{k+1},z^k,y,\nabla_y f(\xkpone,z^k))$ \label{line:fistap_y}
    \STATE $\tkpone \gets \frac{1+\sqrt{1+4t_k^2}}{2}$
    \STATE $z^{k+1} \gets \ybarkpone + \left( \frac{t_k - 1}{\tkpone} \right)(\ybarkpone - \ybar^{k})$
\ENDFOR
\RETURN $(\xKpone,\ybarKpone)$
\end{algorithmic}
\end{algorithm}
The iteration complexity of FISTA-p (and FAPLM-S) is given by the following theorem.

\begin{thm}
Assuming that $\nabla_y f(\cdot)$ is Lipschitz continuous with Lipschitz constant $L_y(f)$ and $\rho \leq \frac{1}{L_y(f)}$, the sequence $\{x^k,\ybar^k\}$ generated by Algorithm \ref{alg:fista_partial} satisfies
\begin{equation}\label{eq:thm_fista_complexity}
    F(x^k,\ybar^k) - F(x^*,y^*) \leq \frac{2\|\ybar^0 - y^*\|^2}{\rho (k+1)^2},
\end{equation}
\end{thm}

Although we need to solve a linear system in every iteration of Algorithms \ref{alg:alms_partial2}, \ref{alg:ista_partial}, and \ref{alg:fista_partial}, the left-hand-side of the system stays constant throughout the invocation of the algorithms because, following Remark \ref{rem:rho_mu}, we can always set $\rho = \mu$.  Hence, no line-search is necessary, and this step essentially requires only one backward- and one forward-substitution, the complexity of which is the same as a gradient step.

\subsection{ISTA/FISTA: full linearization}\label{sec:fista_full}
ISTA solves the following problem in each iteration to produce the next iterate $\left(
                                                                                       \begin{array}{c}
                                                                                         x^+ \\
                                                                                         y^+ \\
                                                                                       \end{array}
                                                                                     \right)
$.
\begin{eqnarray}
  \nonumber \min_{x^\prime,y^\prime} && \frac{1}{2\rho} \left\| \left(
                                         \begin{array}{c}
                                           x^\prime \\
                                           y^\prime \\
                                         \end{array}
                                       \right) - d
   \right\|^2 + \glassonormy \\
   &\equiv& \frac{1}{2\rho}\|x^\prime - d_x\|^2 + \sum_{j}\frac{1}{2\rho}\left( \|y^\prime_j - d_{y_j}\|^2 + \lambda\|y^\prime_j\| \right), \label{eq:fista}
\end{eqnarray}
where $d = \left(
                                          \begin{array}{c}
                                            d_x \\
                                            d_y \\
                                          \end{array}
                                        \right)
        = \left(
            \begin{array}{c}
              x \\
              y \\
            \end{array}
          \right) - \rho\nabla f(x,y)
$, and $f(x,y)$ is defined in \eqref{eq:f_x_ybar}.
It is easy to see that we can solve for $x^+$ and $y^+$ separately in \eqref{eq:fista}.  Specifically,
\begin{eqnarray}
  x^+ &=& d_x \label{eq:fista_x}\\
  y^+_j &=& \frac{d_{y_j}}{\|d_{y_j}\|}\max(0,\|d_{y_j}\|-\lambda\rho), \quad j = 1,\ldots,J.
\end{eqnarray}
Using ISTA to solve the outer augmented Lagrangian \eqref{eq:aug_lag} subproblem is equivalent to taking only skipping steps in ALM-S.  In our experiments, we used the accelerated version of ISTA, i.e. FISTA (Algorithm \ref{alg:fista_full}) to solve \eqref{eq:aug_lag}.

\begin{algorithm}
\caption{FISTA}
\begin{algorithmic}[1]\label{alg:fista_full}
\STATE Given $\xbar^0, \ybar^0, v$. Choose $\rho^0$.  Set $t_0 = 0, z_x^0 = \xbar^0, z_y^0 = \ybar^0$.  Define $f(x,y)$ as in \eqref{eq:f_x_ybar}.
\FOR{$k = 0,1,\cdots$ until stopping criterion is satisfied}
    \STATE Perform a backtracking line-search on $\rho$, starting from $\rho_0$.
    \STATE $\left(
                                          \begin{array}{c}
                                            d_x \\
                                            d_y \\
                                          \end{array}
                                        \right)
        = \left(
            \begin{array}{c}
              z_x^k \\
              z_y^k \\
            \end{array}
          \right) - \rho\nabla f(z_x^k,z_y^k)$
    \STATE $\xbarkpone \gets d_x$
    \STATE $\ybar_j^{k+1} \gets \frac{d_{y_j}}{\|d_{y_j}\|}\max(0,\|d_{y_j}\|-\lambda\rho), \quad j = 1,\ldots,J.$
    \STATE $t_{k+1} \gets \frac{1+\sqrt{1+4t_k^2}}{2}$
    \STATE $z_x^{k+1} \gets \xbar^k + \frac{t_k - 1}{t_{k+1}}(\xbarkpone - \xbar^k)$
    \STATE $z_y^{k+1} \gets \ybar^k + \frac{t_k - 1}{t_{k+1}}(\ybarkpone - \ybar^k)$
\ENDFOR
\RETURN $(\xbar^{K+1},\ybar^{K+1})$
\end{algorithmic}
\end{algorithm}

FISTA (resp. ISTA) is, in fact, an inexact version of FISTA-p (resp. ISTA-p), where we minimize with respect to $x$ a linearized approximation
\begin{equation*}
    \tilde{f}(x,z^k) := f(x^k,z^k) + \nabla_x f(x^k,z^k)(x-x^k) + \frac{1}{2\rho}\|x-x^k\|^2
\end{equation*}
of the quadratic objective function $f(x,z^k)$ in \eqref{eq:fista_partial_x_prob}.  The update to $x$ in Line \ref{line:fistap_x} of Algorithm \ref{alg:fista_partial} is replaced by \eqref{eq:fista_x} as a result.  Similar to FISTA-p, FISTA is also a special skipping version of the full-split FALM-S.  Considering that FISTA has an iteration complexity of $O(\frac{1}{k^2})$, it is not surprising that FISTA-p has the same iteration complexity.

\begin{rem}
Since FISTA requires only the gradient of $f(x,y)$, it can easily handle any smooth convex loss function, such as the logistic loss for binary classification, $L(x) = \sum_{i=1}^{N}\log(1+\exp(-b_i a_i^Tx))$, where $a_i^T$ is the $i$-th row of $A$, and $b$ is the vector of labels. Moreover, when the scale of the data $(\min\{n,m\})$ is so large that it is impractical to compute the Cholesky factorization of $A^TA$, FISTA is a good choice to serve as the subroutine ApproxAugLagMin$(x,y,v)$ in OGLasso-AugLag.
\end{rem}


\section{Overlapping Group $l_1/l_\infty$-Regularization}\label{sec:L1infty}
The subproblems with respect to $y$ (or $\ybar$) involved in all the algorithms presented in the previous sections take the following form
\begin{equation}\label{eq:mult_task_lasso}
    \min_y \frac{1}{2\rho}\|c - y\|^2 + \tilde{\Omega}(y),
\end{equation}
where $\tilde{\Omega}(y) = \lambda\sum_{s\in \tilde{\mathcal{S}}}w_s\|y_s\|_\infty$ in the case of $l_1/l_\infty$-regularization.  In \eqref{eq:adal_y_prob}, for example, $c = Cx - \mu v$.  The solution to \eqref{eq:mult_task_lasso} is the proximal operator of $\tilde{\Omega}$ \cite{combettes2006signal,combettes2009proximal}. Similar to the classical Group Lasso, this problem is block-separable and hence all blocks can be solved simultaneously.

Again, for notational simplicity, we assume $w_s = 1 \quad \forall s \in \tilde{\mathcal{S}}$ and omit it from now on.  For each $s\in \tilde{\mathcal{S}}$, the subproblem in \eqref{eq:mult_task_lasso} is of the form
\begin{equation}\label{eq:l1inf_subprob}
    \min_{y_s} \quad \frac{1}{2}\|c_s - y_s\|^2 + \rho\lambda\|y_s\|_\infty.
\end{equation}
As shown in \cite{sparsa}, the optimal solution to the above problem is $c_s - P(c_s)$, where $P$ denotes the orthogonal projector onto the ball of radius $\rho\lambda$ in the dual norm of the $l_\infty$-norm, i.e. the $l_1$-norm.  The Euclidean projection onto the simplex can be computed in (expected) linear time \cite{duchi2008efficient, brucker1984n}. Duchi et al. \cite{duchi2008efficient} show that the problem of computing the Euclidean projection onto the $l_1$-ball can be reduced to that of finding the Euclidean projection onto the simplex in the following way.  First, we replace $c_s$ in problem \eqref{eq:l1inf_subprob} by $|c_s|$, where the absolute value is taken component-wise.  After we obtain the projection  $z_s$ onto the simplex, we can construct the projection onto the $l_1$-ball by setting $y_s^* = sign(c_s)z_s$, where $sign(\cdot)$ is also taken component-wise.

\section{Experiments}
We tested the OGLasso-AugLag framework (Algorithm \ref{alg:aug_lag}) with four subroutines: ADAL, APLM-S, FISTA-p, and FISTA.  We implemented the framework with the first three subroutines in C++ to compare them with the ProxFlow algorithm proposed in \cite{obozinski2010network}.  We used the C interface and BLAS and LAPACK subroutines provided by the AMD Core Math Library (ACML)\footnote{http://developer.amd.com/libraries/acml/pages/default.aspx.  Ideally, we should have used the Intel Math Kernel Library (Intel MKL), which is optimized for Intel processors, but Intel MKL is not freely available.}.  To compare with ProxGrad \cite{chen2010efficient}, we implemented the framework and all four algorithms in Matlab.  We did not include ALM-S in our experiments because it is time-consuming to find the right $\rho$ for the inner loops as discussed in Remark \ref{rem:rho_mu_alms_full}, and our preliminary computational experience showed that ALM-S was slower than the other algorithms, even when the heuristic $\rho$-setting scheme discussed in Remark \ref{rem:rho_mu_alms_full} was used, because a large number of steps were skipping steps, which meant that the computation involved in solving the linear systems in those steps was wasted.
All of our experiments were performed on a laptop PC with an Intel Core 2 Duo 2.0 GHz processor and 4 Gb of memory.


\subsection{Algorithm parameters and termination criteria}\label{sec:alg_config}
Each algorithm (framework + subroutine)\footnote{For conciseness, we use the subroutine names (e.g. FISTA-p) to represent the full algorithms that consist of the OGLasso-AugLag framework and the subroutines.} required several parameters to be set and termination criteria to be specified.  We used stopping criteria based on the primal and dual residuals as in \cite{boyd2010distributed}. We specify the criteria for each of the algorithms below,  but defer their derivation to Appendix \ref{sec:stopping_criteria}.  The maximum number of outer iterations was set to 500, and the tolerance for the outer loop was set at $\epsilon_{out} = 10^{-4}$.  The number of inner-iterations was capped at 2000, and the tolerance at the $l$-th outer iteration for the inner loops was $\epsilon_{in}^l$.  Our termination criterion for the outer iterations was
\begin{equation}\label{eq:outer_stopping}
    \max\{ r^l, s^l \} \leq \epsilon_{out},
\end{equation}
where $r^l = \frac{\|Cx^l - y^l\|}{\max\{\|Cx^l\|,\|y^l\|\}}$ is the outer relative primal residual and $s^l$ is the relative dual residual, which is given for each algorithm in Table \ref{tab:stopping_criteria}. Recall that $K+1$ is the index of the last inner iteration of the $l$-th outer iteration; for example, for APLM-S, $(x^{l+1},y^{l+1})$ takes the value of the last inner iterate $(x^{K+1},\ybar^{K+1})$.  We stopped the inner iterations when the maximum of the relative primal residual and the relative objective gradient for the inner problem was less than $\epsilon_{in}^l$.  (See Table \ref{tab:stopping_criteria} for the expressions of these two quantities.) We see there that $s^{l+1}$ can be obtained directly from the relative gradient residual computed in the last inner iteration of the $l$-th outer iteration.

\begin{table}
\begin{center}
\begin{tabular}{|c|c|c|c|}
  \hline
  \multirow{2}{*}{Algorithm} & \multirow{2}{*}{Outer rel. dual residual $s^{l+1}$} & \multicolumn{2}{|c|}{Inner iteration}  \\
  \cline{3-4}
  & & Rel. primal residual & Rel. objective gradient residual \\
  \hline
  ADAL & $\frac{\|C^T(y^{l+1} - y^{l})\|}{\|C^T y^{l}\|}$ & - & - \\
  \hline
  FISTA-p & $\frac{\|C^T(\ybar^{K+1}-z^K)\|}{\|C^T z^K\|}$ & $\frac{\|\ybar^{k+1}-z^k\|}{\|z^k\|}$ & $\frac{\|C^T(\ybar^{k+1}-z^k)\|}{\|C^Tz^k\|}$ \\
  \hline
  APLM-S & $\frac{\|C^T(\ybar^{K+1}-y^{K+1})\|}{\|C^T y^{K+1}\|}$ & $\frac{\|\ybar^{k+1} - y^{k+1}\|}{\|y^{k+1}\|}$ & $\frac{\|C^T(\ybar^{k+1} - y^{k+1}\|)}{\|C^T y^{k+1}\|}$ \\
  \hline
  FISTA & $\frac{\left\| \left(
               \begin{array}{c}
                 \xbar^{K+1} \\
                 \ybar^{K+1} \\
               \end{array}
             \right) - \left(
                         \begin{array}{c}
                           z_x^K \\
                           z_y^K \\
                         \end{array}
                       \right)
 \right\|}{\left\|\left(
              \begin{array}{c}
                z_x^K \\
                z_y^K \\
              \end{array}
            \right)\right\|  }$ & $\frac{\left\| \left(
               \begin{array}{c}
                 \xbar^{k+1} \\
                 \ybar^{k+1} \\
               \end{array}
             \right) - \left(
                         \begin{array}{c}
                           z_x^k \\
                           z_y^k \\
                         \end{array}
                       \right)
 \right\|}{\left\|\left(
              \begin{array}{c}
                z_x^k \\
                z_y^k \\
              \end{array}
            \right)\right\|  }$ & $\frac{\left\| \left(
               \begin{array}{c}
                 \xbar^{k+1} \\
                 \ybar^{k+1} \\
               \end{array}
             \right) - \left(
                         \begin{array}{c}
                           z_x^k \\
                           z_y^k \\
                         \end{array}
                       \right)
 \right\|}{\left\|\left(
              \begin{array}{c}
                z_x^k \\
                z_y^k \\
              \end{array}
            \right)\right\|  }$ \\
  \hline
\end{tabular}
\end{center}
\caption{Specification of the quantities used in the outer and inner stopping criteria.}
\label{tab:stopping_criteria}
\end{table}

We set $\mu_0 = 0.01$ in all algorithms except that we set $\mu_0 = 0.1$ in ADAL for the data sets other than the first synthetic set and the breast cancer data set.  We set $\rho = \mu$ in FISTA-p and APLM-S and $\rho_0 = \mu$ in FISTA.

For Theorem \ref{thm:ial_conv} to hold, the solution returned by the function ApproxAugLagMin$(x,y,v)$ has to become increasingly more accurate over the outer iterations.  However, it is not possible to evaluate the sub-optimality quantity $\alpha^l$ in \eqref{eq:sub_acc_cond} exactly because the optimal value of the augmented Lagrangian $\mathcal{L}(x,y,v^l)$ is not known in advance.  In our experiments, we used the maximum of the relative primal and dual residuals $(\max\{r^l,s^l\})$ as a surrogate to $\alpha^l$ for two reasons: First, it has been shown in \cite{boyd2010distributed} that $r^l$ and $s^l$ are closely related to $\alpha^l$.  Second, the quantities $r^l$ and $s^l$ are readily available as bi-products of the inner and outer iterations.  To ensure that the sequence $\{\epsilon_{in}^{l}\}$ satisfies \eqref{eq:sub_acc_cond}, we basically set:
\begin{equation}\label{eq:eps_in}
    \epsilon_{in}^{l+1} = \beta_{in}\epsilon_{in}^l,
\end{equation}
with $\epsilon_{in}^0 = 0.01$ and $\beta_{in} = 0.5$.  However, since we terminate the outer iterations at $\epsilon_{out} > 0$, it is not necessary to solve the subproblems to an accuracy much higher than the one for the outer loop.  On the other hand, it is also important for $\epsilon_{in}^l$ to decrease to below $\epsilon_{out}$, since $s^l$ is closely related to the quantities involved in the inner stopping criteria.  Hence, we slightly modified \eqref{eq:eps_in} and used $\epsilon_{in}^{l+1} = \max\{\beta_{in}\epsilon_{in}^l, 0.2\epsilon_{out}\}$.

Recently, we became aware of an alternative `relative error' stopping criterion \cite{eckstein2011practical} for the inner loops, which guarantees convergence of Algorithm \ref{alg:aug_lag}.  In our context, this criterion essentially requires that the absolute dual residual is less than a fraction of the absolute primal residual.  For FISTA-p, for instance, this condition requires that the $(l+1)$-th iterate satisfies
\begin{equation*}
    2\left\|\left(
                     \begin{array}{c}
                       w_x^0 - x^{l+1} \\
                       w_y^l - y^{l+1} \\
                     \end{array}
                   \right)\right\|\bar{s}^{l+1} + \frac{(\bar{s}^{l+1})^2}{\mu^2} \leq \sigma(\bar{r}^{l+1})^2,
\end{equation*}
where $\bar{r}$ and $\bar{s}$ are the numerators in the expressions for $r$ and $s$ respectively, $\sigma = 0.99$, $w_x^0$ is a constant, and $w_y$ is an auxiliary variable updated in each outer iteration by $w_y^{l+1} = w_y^l - \frac{1}{\mu^2}C^T(\ybar^{K+1}-z^K)$.  We experimented with this criterion but did not find any computational advantage over the heuristic based on the relative primal and dual residuals.

\subsection{Strategies for updating $\mu$}
The penalty parameter $\mu$ in the outer augmented Lagrangian \eqref{eq:aug_lag} not only controls the infeasibility in the constraint $Cx = y$, but also serves as the step-length in the $y$-subproblem (and the $x$-subproblem in the case of FISTA).  We adopted two kinds of strategies for updating $\mu$.  The first one simply kept $\mu$ fixed.  In this case, choosing an appropriate $\mu_0$ was important for good performance.  This was especially true for ADAL in our computational experiments.  Usually, a $\mu_0$ in the range of $10^{-1}$ to $10^{-3}$ worked well.

The second strategy is a dynamic scheme based on the values $r^l$ and $s^l$ \cite{boyd2010distributed}. Since $\frac{1}{\mu}$ penalizes the primal infeasibility, a small $\mu$ tends to result in a small primal residual.  On the other hand, a large $\mu$ tends to yield a small dual residual.  Hence, to keep $r^l$ and $s^l$ approximately balanced in each outer iteration, our scheme updated $\mu$ as follows:
\begin{equation}\label{eq:dyn_mu}
    \mu^{l+1} \leftarrow \left\{
                           \begin{array}{ll}
                             \max\{\beta\mu^l, \mu_{min}\}, & \hbox{if $r^l > \tau s^l$} \\
                             \min\{\mu^l / \beta, \mu_{max}\}, & \hbox{if $s^l > \tau r^l$} \\
                             \mu^l, & \hbox{otherwise,}
                           \end{array}
                         \right.
\end{equation}
where we set $\mu_{max} = 10$, $\mu_{min} = 10^{-6}$, $\tau = 10$ and $\beta = 0.5$, except for the first synthetic data set, where we set $\beta = 0.1$ for ADAL, FISTA-p, and APLM-S.


\subsection{Synthetic examples}
To compare our algorithms with the ProxGrad algorithm proposed in \cite{chen2010efficient}, we first tested a synthetic data set (ogl) using the procedure reported in \cite{chen2010efficient} and \cite{jacob2009group}.  The sequence of decision variables $x$ were arranged in groups of ten, with adjacent groups having an overlap of three variables.  The support of $x$ was set to the first half of the variables.  Each entry in the design matrix $A$ and the non-zero entries of $x$ were sampled from i.i.d. standard Gaussian distributions, and the output $b$ was set to $b = Ax + \epsilon$, where the noise $\epsilon \sim \mathcal{N}(0,I)$.  Two sets of data were generated as follows: (a) Fix $n=5000$ and vary the number of groups $J$ from 100 to 1000 with increments of 100.  (b) Fix $J=200$ and vary $n$ from 1000 to 10000 with increments of 1000.  The stopping criterion for ProxGrad was the same as the one used for FISTA, and we set its smoothing parameter to $10^{-3}$.  Figure \ref{fig:chen_plot} plots the CPU times taken by the Matlab version of our algorithms and ProxGrad (also in Matlab) on theses scalability tests on $l_1/l_2$-regularization.
A subset of the numerical results on which these plots are based is presented in Tables \ref{tab:chen_data1} and \ref{tab:chen_data2}.

The plots clearly show that the alternating direction methods were much faster than ProxGrad on these two data sets.  Compared to ADAL, FISTA-p performed slightly better, while it showed obvious computational advantage over its general version APLM-S.  In the plot on the left of Figure \ref{fig:chen_plot}, FISTA exhibited the advantage of a gradient-based algorithm when both $n$ and $m$ are large.  In that case (towards the right end of the plot), the Cholesky factorizations required by ADAL, APLM-S, and FISTA-p became relatively expensive.  When $\min\{n,m\}$ is small or the linear systems can be solved cheaply, as the plot on the right shows, FISTA-p and ADAL have an edge over FISTA due to the smaller numbers of inner iterations required.


We generated a second data set (dct) using the approach from \cite{obozinski2010network} for scalability tests on both the $l_1/l_2$ and $l_1/l_\infty$ group penalties.  The design matrix $A$ was formed from over-complete dictionaries of discrete cosine transforms (DCT).  The set of groups were all the contiguous sequences of length five in one-dimensional space.  $x$ had about $10\%$ non-zero entries, selected randomly.  We generated the output as $b = Ax + \epsilon$, where $\epsilon \sim \mathcal{N}(0,0.01\|Ax\|^2)$.  We fixed $n = 1000$ and varied the number of features $m$ from 5000 to 30000 with increments of 5000.  This set of data leads to considerably harder problems than the previous set because the groups are heavily overlapping, and the DCT dictionary-based design matrix exhibits local correlations.  Due to the excessive running time required on Matlab, we ran the C++ version of our algorithms for this data set, leaving out APLM-S and ProxGrad, whose performance compared to the other algorithms is already fairly clear from Figure \ref{fig:chen_plot}.  For ProxFlow, we set the tolerance on the relative duality gap to $10^{-4}$, the same as $\epsilon_{out}$, and kept all the other parameters at their default values.


Figure \ref{fig:tony_plot} presents the CPU times required by the algorithms versus the number of features.  In the case of $l_1/l_2$-regularization, it is clear that FISTA-p outperformed the other two algorithms.  For $l_1/l_\infty$-regularization, ADAL and FISTA-p performed equally well and compared favorably to ProxFlow.  In both cases, the growth of the CPU times for FISTA follows the same trend as that for FISTA-p, and they required a similar number of outer iterations, as shown in Tables \ref{tab:dct_scal_results_l12} and \ref{tab:dct_scal_results_l1inf}.  However, FISTA lagged behind in speed due to larger numbers of inner iterations.  Unlike in the case of the ogl data set, Cholesky factorization was not a bottleneck for FISTA-p and ADAL here because we needed to compute it only once.

To simulate the situation where computing or caching $A^TA$ and its Cholesky factorization is not feasible, we switched ADAL and FISTA-p to PCG mode by always using PCG to solve the linear systems in the subproblems.  We compared the performance of ADAL, FISTA-p, and FISTA on the previous data set for both $l_1/l_2$ and $l_1/l_\infty$ models.  The results for ProxFlow are copied from from Figure \ref{fig:tony_plot} and Table \ref{tab:tony_pcg_l1inf} to serve as a reference.  We experimented with the fixed-value and the dynamic updating schemes for $\mu$ on all three algorithms.  From Figure \ref{fig:tony_pcg_plot}, it is clear that the performance of FISTA-p was significantly improved by using the dynamic scheme.  For ADAL, however, the dynamic scheme worked well only in the $l_1/l_2$ case, whereas the performance turned worse in general in the $l_1/l_\infty$ case.  We did not include the results for FISTA with the dynamic scheme because the solutions obtained were considerably more suboptimal than the ones obtained with the fixed-$\mu$ scheme.  Tables \ref{tab:tony_pcg_l12} and \ref{tab:tony_pcg_l1inf} report the best results of the algorithms in each case. The plots and numerical results show that FISTA-p compares favorably to ADAL and stays competitive to ProxFlow. In terms of the quality of the solutions, FISTA-p and ADAL also did a better job than FISTA, as evidenced in Table \ref{tab:tony_pcg_l1inf}.  On the other hand, the gap in CPU time between FISTA and the other three algorithms is less obvious.

\begin{figure}
    \hspace*{-0.0in}\includegraphics[width=1.0\textwidth]{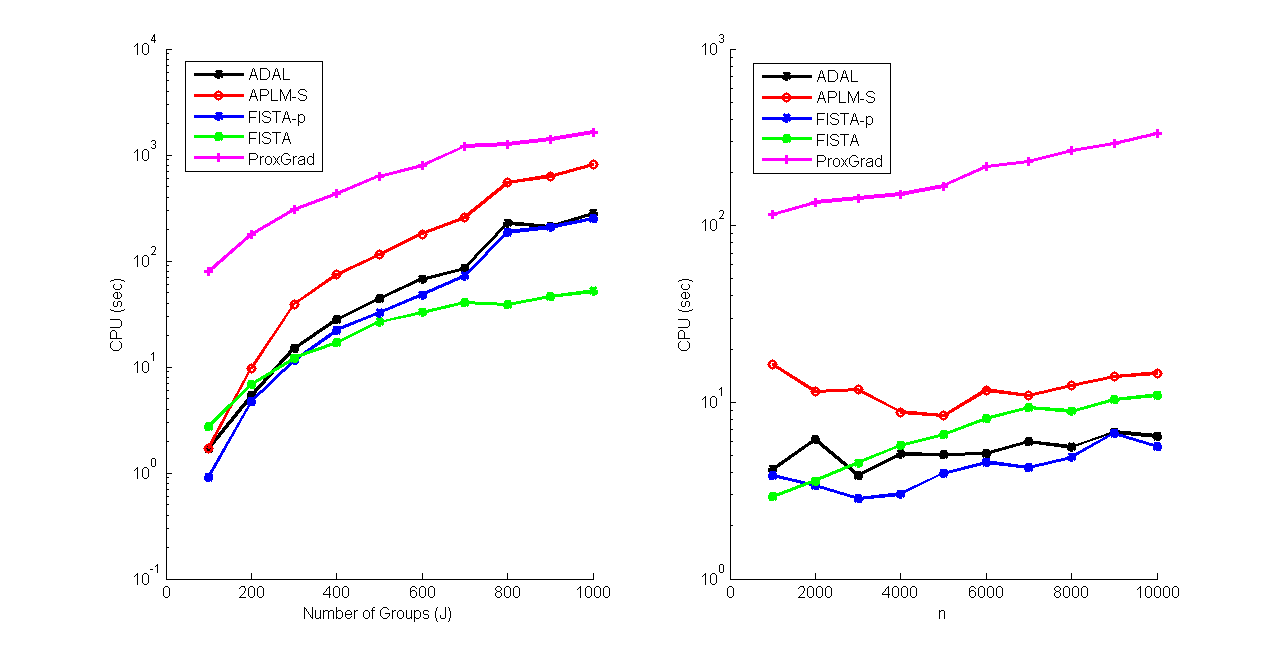}
    \caption{Scalability test results of the algorithms on the synthetic overlapping Group Lasso data sets from \cite{chen2010efficient}.  The scale of the $y$-axis is logarithmic. The dynamic scheme for $\mu$ was used for all algorithms except ProxGrad.}
    \label{fig:chen_plot}
\end{figure}


\begin{figure}
    \hspace*{-0.0in}\includegraphics[width=1.0\textwidth]{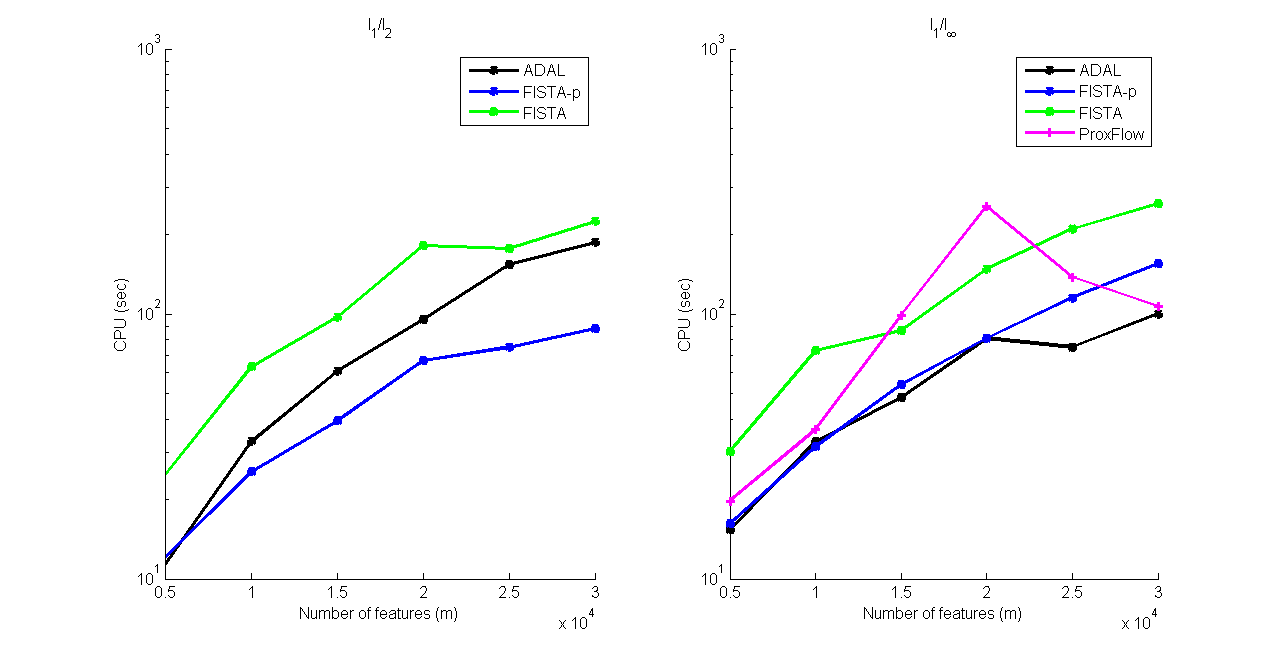}
    \caption{Scalability test results on the DCT set with $l_1/l_2$-regularization (left column) and $l_1/l_\infty$-regularization (right column).  The scale of the $y$-axis is logarithmic.  All of FISTA-p, FITSA, and ADAL were run with a fixed $\mu = \mu_0$.}
    \label{fig:tony_plot}
\end{figure}

\begin{figure}
    \hspace*{-0.0in}\includegraphics[width=1.0\textwidth]{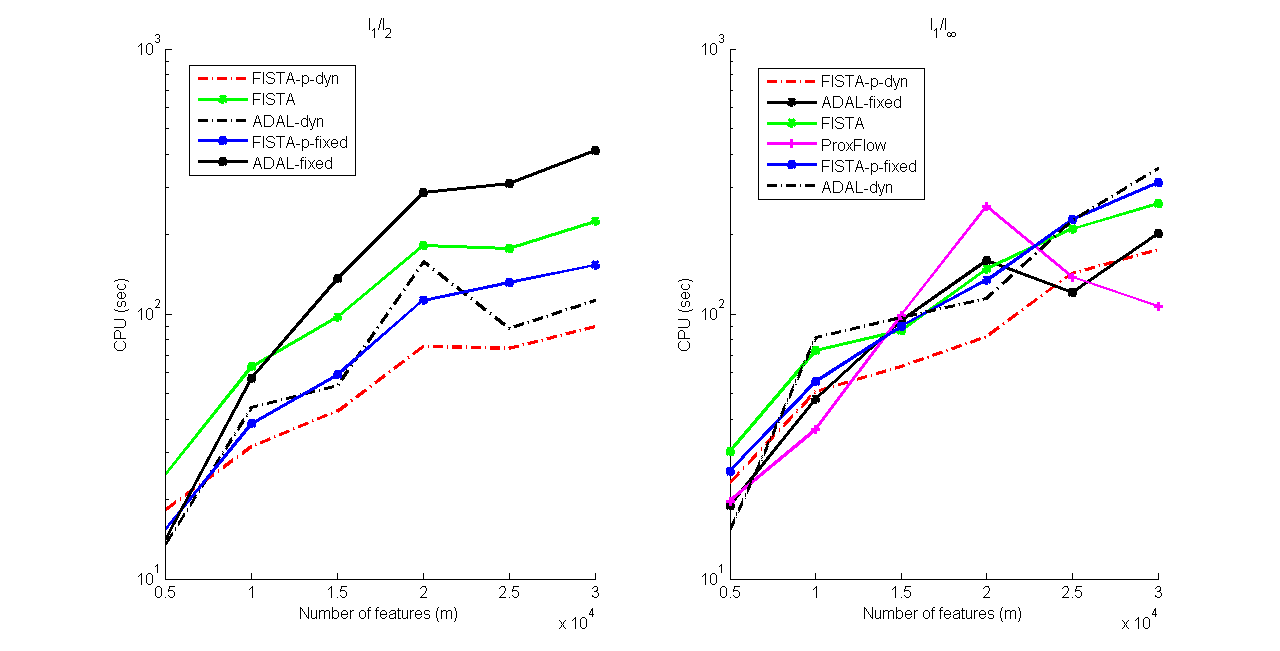}
    \caption{Scalability test results on the DCT set with $l_1/l_2$-regularization (left column) and $l_1/l_\infty$-regularization (right column).  The scale of the $y$-axis is logarithmic.  FISTA-p and ADAL are in PCG mode. The dotted lines denote the results obtained with the dynamic updating scheme for $\mu$. }
    \label{fig:tony_pcg_plot}
\end{figure}

\subsection{Real-world Examples}
To demonstrate the practical usefulness of our algorithms, we tested our algorithms on two real-world applications.

\subsubsection{Breast Cancer Gene Expressions}
We used the breast cancer data set \cite{van2002gene} with canonical pathways from MSigDB \cite{subramanian2005gene}.  The data was collected from 295 breast cancer tumor samples and contains gene expression measurements for 8,141 genes.  The goal was to select a small set of the most relevant genes that yield the best prediction performance.  A detailed description of the data set can be found in \cite{chen2010efficient, jacob2009group}.  In our experiment, we performed
a regression task to predict the length of survival of the patients.  The canonical pathways naturally provide grouping information of the genes.  Hence, we used them as the groups for the group-structured regularization term $\Omega(\cdot)$.

Table \ref{tab:breast_cancer_data} summarizes the data attributes.  The numerical results for the $l_1/l_2$-norm are collected in Table \ref{tab:breast_cancer_results}, which show that FISTA-p and ADAL were the fastest on this data set.  Again, we had to tune ADAL with different initial values $(\mu_0)$ and updating schemes of $\mu$ for speed and quality of the solution, and we eventually kept $\mu$ constant at 0.01.  The dynamic updating scheme for $\mu$ also did not work for FISTA, which returned a very suboptimal solution in this case.  We instead adopted a simple scheme of decreasing $\mu$ by half every 10 outer iterations.  Figure \ref{fig:obj_comp} graphically depicts the performance of the different algorithms.  In terms of the outer iterations, APLM-S behaved identically to FISTA-p, and FISTA also behaved similarly to ADAL.  However, APLM-S and FISTA were considerably slower due to larger numbers of inner iterations.  

We plot the root-mean-squared-error (RMSE) over different values of $\lambda$ (which lead to different numbers of active genes) in the left half of Figure \ref{fig:breast_cancer_plots}.  The training set consists of 200 randomly selected samples, and the RMSE was computed on the remaining 95 samples.  $l_1/l_2$-regularization achieved lower RMSE in this case.  However, $l_1/l_\infty$-regularization yielded better group sparsity as shown in Figure \ref{fig:bc_sparsities}.  The sets of active genes selected by the two models were very similar as illustrated in the right half of Figure \ref{fig:breast_cancer_plots}.  In general, the magnitudes of the coefficients returned by $l_1/l_\infty$-regularization tended to be similar within a group, whereas those returned by $l_1/l_2$-regularization did not follow that pattern.  This is because $l_1/l_\infty$-regularization penalizes only the maximum element, rather than all the coefficients in a group, resulting in many coefficients having the same magnitudes.  

\begin{table}
\begin{tabular}{|c|c|c|c|c|}
  \hline
  Data sets & N (no. samples) & J (no. groups) & group size & average frequency  \\
  \hline
  BreastCancerData & 295 & 637 & 23.7 (avg) & 4 \\
  \hline
\end{tabular}
\caption{The Breast Cancer Dataset}
\label{tab:breast_cancer_data}
\end{table}

\begin{figure}
    \hspace*{-0.1in}\includegraphics[width=0.5\textwidth]{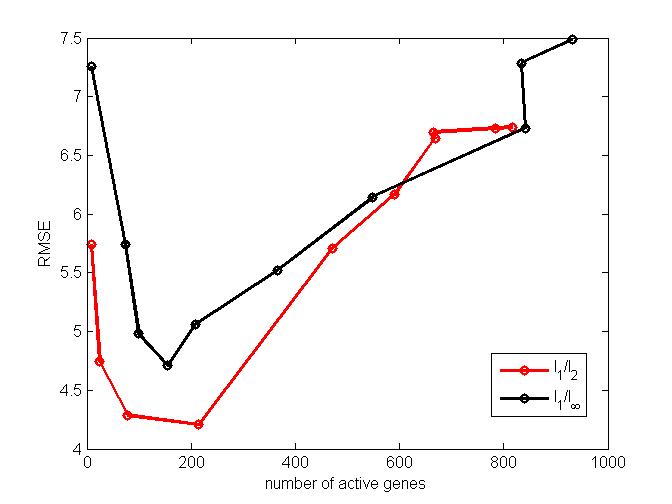}
    \hspace*{-0.1in}\includegraphics[width=0.5\textwidth]{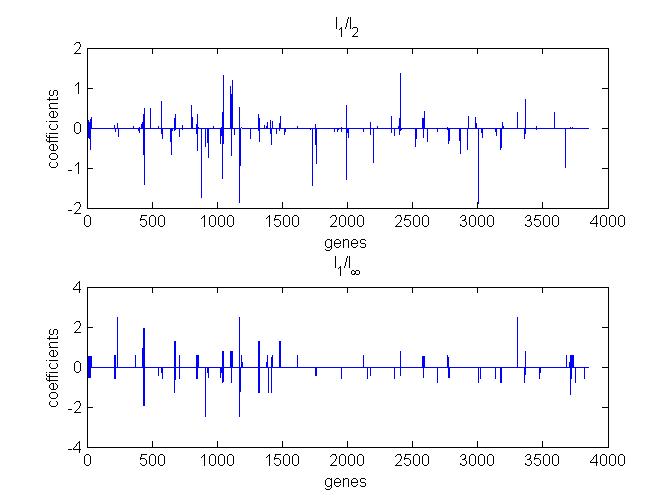}
    \caption{On the left: Plot of root-mean-squared-error against the number of active genes for the Breast Cancer data.  The plot is based on the regularization path for ten different values for $\lambda$.  The total CPU time (in Matlab) using FISTA-p was 51 seconds for $l_1/l_2$-regularization and 115 seconds for $l_1/l_\infty$-regularization.  On the right: The recovered sparse gene coefficients for predicting the length of the survival period. The value of $\lambda$ used here was the one minimizing the RMSE in the plot on the left.}
    \label{fig:breast_cancer_plots}
\end{figure}

\begin{figure}
    \hspace*{1.0in}\includegraphics[width=0.50\textwidth]{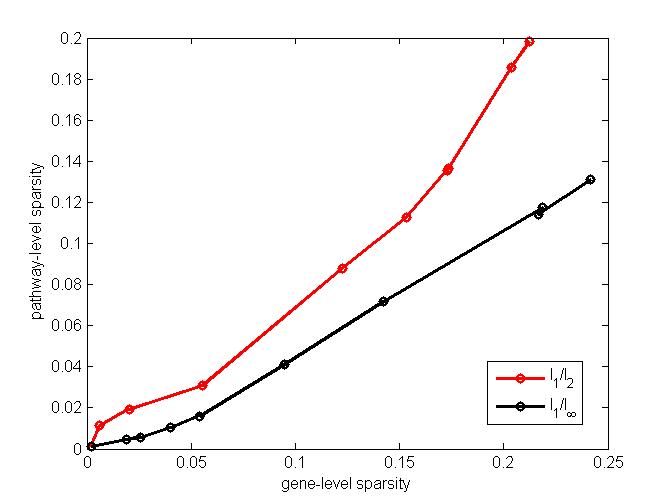}
    \caption{Pathway-level sparsity v.s. Gene-level sparsity.}
    \label{fig:bc_sparsities}
\end{figure}

\begin{figure}
    \hspace*{-0.1in}\includegraphics[width=0.50\textwidth]{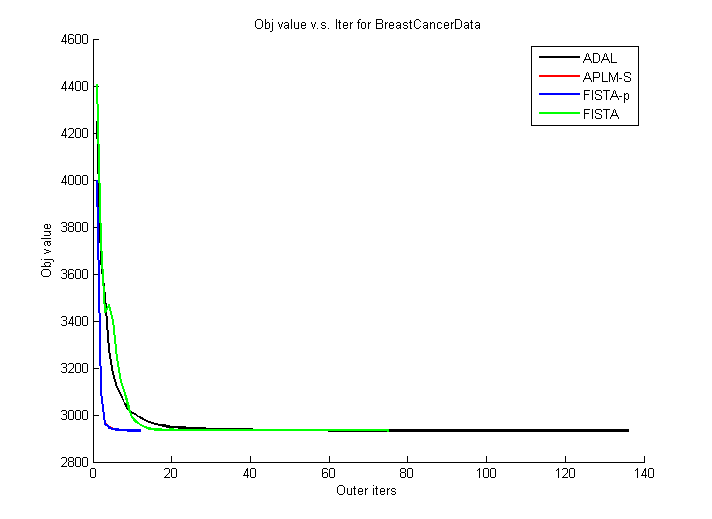}
    \hspace*{-0.1in}\includegraphics[width=0.50\textwidth]{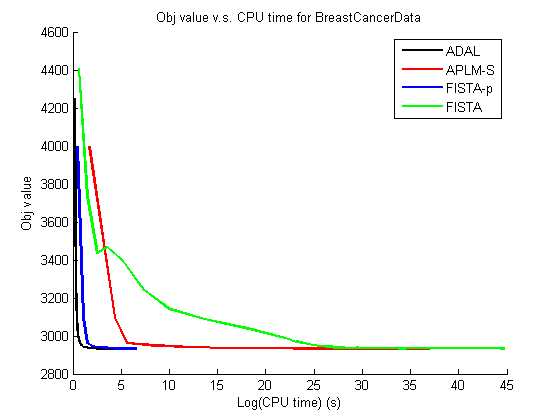}
    \caption{Objective values v.s. Outer iters and Objective values v.s. CPU time plots for the Breast Cancer data.  The results for ProxGrad are not plotted due to the different objective function that it minimizes.  The red (APLM-S) and blue (FISTA-p) lines overlap in the left column.}
    \label{fig:obj_comp}
\end{figure}

\subsubsection{Video Sequence Background Subtraction}
We next considered the video sequence background subtraction task from \cite{obozinski2010network,huang2009learning}.  The main objective here is to segment out foreground objects in an image (frame), given a sequence of $m$ frames from a fixed camera.  The data used in this experiment is available online \footnote{http://research.microsoft.com/en-us/um/people/jckrumm/wallflower/testimages.htm} \cite{toyama1999wallflower}.  The basic setup of the problem is as follows.  We represent each frame of $n$ pixels as a column vector $A_j \in \mathbb{R}^n$ and form the matrix $A \in \mathbb{R}^{n\times m}$ as $A \equiv \left(
                                                                                \begin{array}{cccc}
                                                                                  A_1 & A_2 & \cdots & A_m \\
                                                                                \end{array}
                                                                              \right)
$. The test frame is represented by $b \in \mathbb{R}^n$.  We model the relationship between $b$ and $A$ by $b \approx Ax + e$, where $x$ is assumed to be sparse, and $e$ is the 'noise' term which is also assumed to be sparse.  $Ax$ is thus a sparse linear combination of the video frame sequence and accounts for the background present in both $A$ and $b$.  $e$ contains the sparse foreground objects in $b$.  The basic model with $l_1$-regularization (Lasso) is
\begin{equation}\label{eq:video_L1}
    \min_{x,e} \frac{1}{2}\|Ax+e-b\|^2 + \lambda(\|x\|_1 + \|e\|_1).
\end{equation}
It has been shown in \cite{obozinski2010network} that we can significantly improve the quality of segmentation by applying a group-structured regularization $\Omega(\cdot)$ on $e$, where the groups are all the overlapping $k\times k$-square patches in the image.  Here, we set $k=3$.  The model thus becomes
\begin{equation}\label{eq:video_ogl}
    \min_{x,e} \frac{1}{2}\|Ax+e-b\|^2 + \lambda(\|x\|_1 + \|e\|_1 + \Omega(e)).
\end{equation}
Note that \eqref{eq:video_ogl} still fits into the group-sparse framework if we treat the $l_1$-regularization terms as the sum of the group norms, where the each groups consists of only one element.

We also considered an alternative model, where a Ridge regularization is applied to $x$ and an Elastic-Net penalty \cite{zou2005regularization} to $e$.  This model
\begin{equation}\label{eq:video_ridge_en}
    \min_{x,e} \frac{1}{2}\|Ax+e-b\|^2 + \lambda_1\|e\|_1 + \lambda_2(\|x\|^2 + \|e\|^2)
\end{equation}
does not yield a sparse $x$, but sparsity in $x$ is not a crucial factor here. It
is, however, well suited for our partial linearization methods (APLM-S and FISTA-p), since there is no need for the augmented Lagrangian framework.  Of course, we can also apply FISTA to solve \eqref{eq:video_ridge_en}.

We recovered the foreground objects by solving the above optimization problems and applying the sparsity pattern of $e$ as a mask for the original test frame.  A hand-segmented evaluation image from \cite{toyama1999wallflower} served as the ground truth.  The regularization parameters $\lambda, \lambda_1$, and $\lambda_2$ were selected in such a way that the recovered foreground objects matched the ground truth to the maximum extent.

FISTA-p was used to solve all three models.  The $l_1$ model \eqref{eq:video_L1} was treated as a special case of the group regularization model \eqref{eq:video_ogl}, with each group containing only one component of the feature vector. \footnote{We did not use the original version of FISTA to solve the model as an $l_1$-regularization problem because it took too long to converge in our experiments due to extremely small step sizes.}  For the Ridge/Elastic-Net penalty model, we applied FISTA-p directly without the outer augmented Lagrangian layer.

The solutions for the $l_1/l_2, l_1/l_\infty$, and Lasso models were not strictly sparse in the sense that those supposedly zero feature coefficients had non-zero (albeit extremely small) magnitudes, since we enforced the linear constraints $Cx = y$ through an augmented Lagrangian approach.  To obtain sparse solutions, we truncated the non-sparse solutions using thresholds ranging from $10^{-9}$ to $10^{-3}$ and selected the threshold that yielded the best accuracy.

Note that because of the additional feature vector $e$, the data matrix is effectively $\tilde{A} = \left(
                                                                        \begin{array}{cc}
                                                                          A & I_n \\
                                                                        \end{array}
                                                                      \right) \in \mathbb{R}^{n\times (m+n)}
$.  For solving \eqref{eq:video_ogl}, FISTA-p has to solve the linear system
\begin{equation}\label{eq:video_ls_struct}
    \left(
      \begin{array}{cc}
        A^TA+\frac{1}{\mu}D_x & A^T \\
        A & I_n+\frac{1}{\mu}D_e \\
      \end{array}
    \right)
    \left(
      \begin{array}{c}
        x \\
        e \\
      \end{array}
    \right)  =
    \left(
      \begin{array}{c}
        r_x \\
        r_e \\
      \end{array}
    \right),
\end{equation}
where $D$ is a diagonal matrix, and $D_x, D_e, r_x, r_e$ are the components of $D$ and $r$ corresponding to $x$ and $e$ respectively.
In this example, $n$ is much larger than $m$, e.g. $n = 57600, m = 200$.  To avoid solving a system of size $n\times n$, we took the Schur complement of $I_n + \frac{1}{\mu}D_e$ and solved instead the positive definite $m\times m$ system
\begin{eqnarray}
    \left( A^TA + \frac{1}{\mu}D_x - A^T(I+\frac{1}{\mu}D_e)^{-1}A \right)x &=& r_x - A^T(I+\frac{1}{\mu}D_e)^{-1}r_e, \label{eq:video_ls_small} \\
    e &=& diag(\mathbf{1} + \frac{1}{\mu}D_e)^{-1}(r_e - Ax).
\end{eqnarray}

The $l_1/l_\infty$ model yielded the best background separation accuracy (marginally better than the $l_1/l_2$ model), but it also was the most computationally expensive. (See Table \ref{tab:bg_subt_results} and Figure \ref{fig:bg_subt_figs}.)  Although the Ridge/Elastic-Net model yielded as poor separation results as the Lasso ($l_1$) model, it was orders of magnitude faster to solve using FISTA-p.  We again observed that the dynamic scheme for $\mu$ worked better for FISTA-p than for ADAL.  For a constant $\mu$ over the entire run, ADAL took at least twice as long as FISTA-p to produce a solution of the same quality.  A typical run of FISTA-p on this problem with the best selected $\lambda$ took less than 10 outer iterations.  On the other hand, ADAL took more than 500 iterations to meet the stopping criteria.


\begin{figure}
    \hspace*{-0.0in}\includegraphics[width=1.0\textwidth]{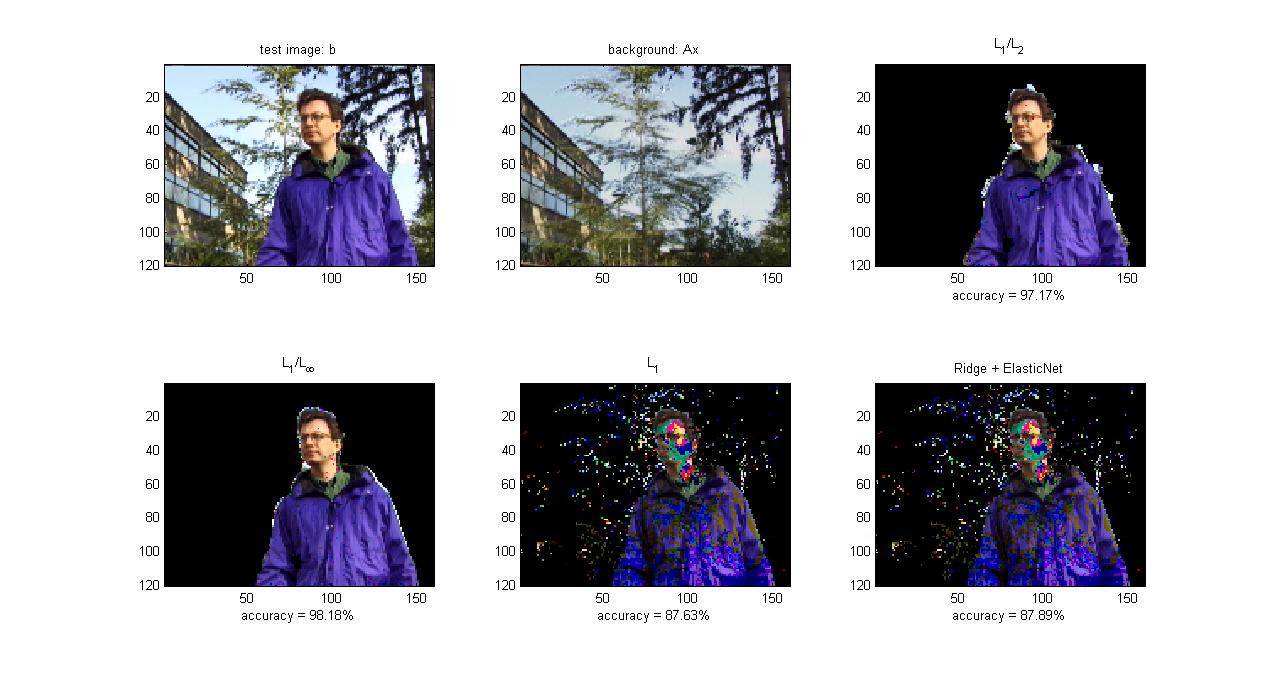}
    \caption{Separation results for the video sequence background substraction example.  Each training image had $120\times 160$ RGB pixels.  The training set contained 200 images in sequence.  The accuracy indicated for each of the different models is the percentage of pixels that matched the ground truth.}
    \label{fig:bg_subt_figs}
\end{figure}

\begin{table}
\begin{tabular}{|c|c|c|c|}
\hline\hline
Model & Accuracy (percent) & Total CPU time (s) & No. parameter values on reg path\\
\hline
$l_1/l_2$& 	 97.17& 	 2.48e+003& 	 8\\
\hline
$l_1/l_\infty$& 	 98.18& 	 4.07e+003& 	 6\\
\hline
     $l_1$& 	 87.63& 	 1.61e+003& 	 11\\
\hline
ridge + elastic net& 	 87.89& 	 1.82e+002& 	 64\\
\hline\hline
\end{tabular}
\caption{Computational results for the video sequence background subtraction example.  The algorithm used is FISTA-p.  We used the Matlab version for the ease of generating the images.  The C++ version runs at least four times faster from our experience in the previous experiments. We report the best accuracy found on the regularization path of each model.  The total CPU time is recorded for computing the entire regularization path, with the specified number of different regularization parameter values.}
\label{tab:bg_subt_results}
\end{table}

\subsection{Comments on Results}
The computational results exhibit two general patterns.  First, the simpler algorithms (FISTA-p and ADAL) were significantly faster than the more general algorithms, such as APLM-S.  Interestingly, the majority of the APLM-S inner iterations consisted of a skipping step for the tests on synthetic data and the breast cancer data, which means that APLM-S essentially behaved like ISTA-p in these cases.  Indeed, FISTA-p generally required the same number of outer-iterations as APLM-S but much fewer inner-iterations, as predicted by theory.  In addition, no computational steps were wasted and no function evaluations were required for FISTA-p and ADAL.  Second, FISTA-p converged faster (required less iterations) than its full-linearization counterpart FISTA.  We have suggested possible reasons for this in Section \ref{sec:partial_linearization}.  On the other hand, FISTA was very effective for data both of whose dimensions were large because it required only gradient computations and soft-thresholding operations, and did not require linear systems to be solved.


Our experiments showed that the performance of ADAL (as well as the quality of the solution that it returned) varied a lot as a function of the parameter settings, and it was tricky to tune them optimally.  In contrast, FISTA-p exhibited fairly stable performance for a simple set of parameters that we rarely had to alter and in general performed better than ADAL.

It may seem straight-forward to apply FISTA directly to the Lasso problem \eqref{eq:video_L1} without the augmented Lagrangian framework.\footnote{To avoid confusion with our algorithms that consist of inner-outer iterations, we prefix our algorithms with `AugLag' here.}  However, as we have seen in our experiments, FISTA took much longer than AugLag-FISTA-p to solve this problem.  We believe that this is further evidence of the `load-balancing' property of the latter algorithm that we discussed in Section \ref{sec:alms_partial2}.  It also demonstrates the versatility of our approach to regularized learning problems.

\section{Conclusion}
We have built a unified framework for solving sparse learning problems involving group-structured regularization, in particular, the $l_1/l_2$- or $l_1/l_\infty$-regularization of arbitrarily overlapping groups of variables.  For the key building-block of this framework, we developed new efficient algorithms based on alternating partial-linearization/splitting, with proven convergence rates.  In addition, we have also incorporated ADAL and FISTA into our framework.  Computational tests on several sets of synthetic test data demonstrated the relative strength of the algorithms, and through two real-world applications we compared the relative merits of these structured sparsity-inducing norms.  Among the algorithms studied, FISTA-p and ADAL performed the best on most of the data sets, and FISTA appeared to be a good alternative choice for large-scale data.  From our experience, FISTA-p is easier to configure and is more robust to variations in the algorithm parameters.  Together, they form a flexible and versatile suite of methods for group-sparse problems of different sizes.

\section{Acknowledgement}
We would like to thank Katya Scheinberg and Shiqian Ma for many helpful discussions, and Xi Chen for providing the Matlab code for ProxGrad.

\appendix
\section{Proof of Lemma \ref{lem:lem22x}}\label{sec:pf_lemma}
\begin{eqnarray}
  \nonumber F(\xbar,\ybar) - F(x,\pfxy) &\geq& F(\xbar,\ybar) - \Lrho(x,y,\pfxy,\nabla_y f(x,y)) \\
  \nonumber &=& F(\xbar,\ybar) - \bigg( f(x,y) + \gradyfxy^T(\pfxy - y) \\
   && + \frac{1}{2\rho}\|\pfxy - y\|^2 + g(\pfxy)  \bigg). \label{eq:207x}
\end{eqnarray}
From the optimality of $\pfxy$, we also have
\begin{equation}\label{eq:210x}
    \gamma_g(\pfxy) + \gradyfxy + \frac{1}{\rho}(\pfxy - y) = 0.
\end{equation}
Since $F(x,y) = f(x,y) + g(y)$, and $f$ and $g$ are convex functions, for any $(\xbar, \ybar)$,
\begin{equation}\label{eq:211x}
    F(\xbar,\ybar) \geq g(\pfxy) + (\ybar - \pfxy)^T\gamma_g(\pfxy) + f(x,y) + (\ybar - y)^T\gradyfxy + (\xbar - x)^T\gradxfxy.
\end{equation}
Therefore, from \eqref{eq:207x}, \eqref{eq:210x}, and \eqref{eq:211x}, it follows that
\begin{eqnarray}
  \nonumber F(\xbar,\ybar) - F(x,\pfxy) &\geq& g(\pfxy) + (\ybar - \pfxy)^T\gamma_g(\pfxy) + f(x,y) + (\ybar - y)^T\gradyfxy + (\xbar - x)^T\gradxfxy \\
  \nonumber && - \left( f(x,y) + \gradyfxy^T(\pfxy - y) + \frac{1}{2\rho}\|\pfxy - y\|^2 + g(\pfxy)  \right) \\
  \nonumber &=& (\ybar - \pfxy)^T(\gamma_g(\pfxy) + \gradyfxy) - \frac{1}{2\rho}\|\pfxy - y\|^2 + (\xbar-x)^T\gradxfxy \\
  \nonumber &=& (\ybar - \pfxy)^T\left( -\frac{1}{\rho}(\pfxy - y) \right) - \frac{1}{2\rho}\|\pfxy - y\|^2 + (\xbar-x)^T\gradxfxy \\
  &=& \frac{1}{2\rho}(\|\pfxy - \ybar\|^2 - \|y-\ybar\|^2) + (\xbar-x)^T\gradxfxy. \label{eq:212x}
\end{eqnarray}
The proof for the second part of the lemma is very similar, but we give it for completeness.
\begin{equation}\label{eq:207xb}
    F(x,y) - F(\pgybar) \geq F(x,y) - \left( f(\pgybar) + \gybar + \gammagxbary^T(\pgybar_y - \ybar) + \frac{1}{2\rho}\|\pgybar_y - \ybar\|^2  \right)
\end{equation}
By the optimality of $\pgybar$, we have
\begin{eqnarray}
  \nabla_x f(\pgybar) &=& 0, \label{eq:210xb1}\\
  \nabla_{y} f(\pgybar) + \gamma_g(\ybar) + \frac{1}{\rho}(\pgybar_y - \ybar) &=& 0. \label{eq:210xb2}
\end{eqnarray}
Since $F(x,y) = f(x,y) + g(y)$, it follows from the convexity of both $f$ and $g$ and \eqref{eq:210xb1} that
\begin{equation}\label{eq:211xb}
    F(x,y) \geq \gybar + (y - \ybar)^T\gammagxbary + f(\pgybar) + (y - \pgybar_y)^T\nabla_{y}f(\pgybar).
\end{equation}
Now combining \eqref{eq:207xb}, \eqref{eq:210xb2}, and \eqref{eq:211xb}, it follows that
\begin{eqnarray}
  \nonumber F(x,y) - F(\pgybar) &\geq& (y - \pgybar_y)^T(\gammagxbary_{\ybar} + \nabla_{y}f(\pgybar)) - \frac{1}{2\rho}\|\pgybar_y - \ybar\|^2 \\
  \nonumber &=& (y - \pgybar_y)^T\left( \frac{1}{\rho}(\ybar - \pgybar_y)  \right) - \frac{1}{2\rho}\|\pgybar_y - \ybar\|^2 \\
   &=& \frac{1}{2\rho}(\|\pgybar_y - y\|^2 - \|y - \ybar\|^2). \label{eq:212xb}
\end{eqnarray}

\section{Proof of Theorem \ref{thm:alms_partial_complexity}}\label{sec:pf_thm}
Let $I$ be the set of all regular iteration indices among the first $k-1$ iterations, and let $I_c$ be its complement.  For all $n \in I_c, y^{n+1}=\ybar^n$.

For $n \in I$, we can apply Lemma \ref{lem:lem22x} since \eqref{eq:cond1_lem22ybar} automatically holds, and  \eqref{eq:cond1_lem22x} holds when $\rho \leq \frac{1}{L(f)}$.  In \eqref{eq:lem22B}, by letting $(x,y) = (x^*, y^*)$, and $\ybar = \ybar^n$, we get $\pgybar = (x^{n+1}, y^{n+1})$, and
\begin{equation}\label{eq:215x}
    2\rho(F(x^*,y^*) - F(x^{n+1}, y^{n+1})) \geq \|y^{n+1}-y^*\|^2 - \|\ybar^n-y^*\|^2.
\end{equation}
In \eqref{eq:lem22A}, by letting $(\xbar, \ybar) = (x^*, y^*), (x,y) = (x^{n+1}, y^{n+1})$, we get $\pfxy = \ybarnpone$ and
\begin{eqnarray}\label{eq:214}
    \nonumber 2\rho(F(x^*, y^*) - F(\xnpone, \ybarnpone) ) &\geq& \|\ybarnpone - y^*\|^2 - \|\ynpone - y^*\|^2 + (x^* - \xnpone)^T\nabla_x f(\xnpone, \ynpone) \\
    &=& \|\ybarnpone - y^*\|^2 - \|\ynpone - y^*\|^2. \label{eq:214x},
\end{eqnarray}
since $\nabla_x f(\xnpone, \ynpone) = 0$, for $n \in I$ by \eqref{eq:210xb1} and for $n \in I_c$ by \eqref{eq:x_skip_opt}.
Adding \eqref{eq:214x} to \eqref{eq:215x}, we get
\begin{equation}\label{eq:216x}
    2\rho( 2F(x^*,y^*) - F(\xnpone,\ynpone) - F(\xnpone,\ybarnpone) ) \geq \|\ybarnpone - y^*\|^2 -\|\ybar^n - y^*\|^2.
\end{equation}
For $n \in I_c$, since $\nabla_x f(\xnpone, \ynpone) = 0$, we have that \eqref{eq:214x} holds.  Since $\ynpone = \ybar^n$, it follows that
\begin{equation}\label{eq:217x}
    2\rho(F(x^*,y^*) - F(\xnpone,\ybarnpone)) \geq \|\ybarnpone - y^*\|^2 - \|\ybar^n - y^*\|^2.
\end{equation}
Summing \eqref{eq:216x} and \eqref{eq:217x} over $n = 0,1,\ldots,k-1$ and observing that $2|I|+|I_c| = k+k_n$, we obtain
\begin{eqnarray}\label{eq:218x}
  && 2\rho\left( (k+k_n)F(x^*,y^*) - \sum_{n=1}^{k-1}F(\xnpone,\ybarnpone) - \sum_{n\in I}F(\xnpone, \ynpone) \right) \\
  \nonumber &\geq& \sum_{n=0}^{k-1}(\|\ybarnpone - y^*\|^2 - \|\ybar^n - y^*\|^2 ) \\
  \nonumber &=& \|\ybar^k - y^*\|^2 - \|\ybar^0 - y^*\|^2 \\
  \nonumber &\geq& -\|\ybar^0 - y^*\|^2.
\end{eqnarray}
In Lemma \ref{lem:lem22x}, by letting $(\xbar,\ybar) = (\xnpone,\ynpone)$ in \eqref{eq:lem22A} instead of $(x^*,y^*)$, we have from \eqref{eq:214x} that
\begin{equation}\label{eq:219x}
    2\rho( F(\xnpone,\ynpone) - F(\xnpone,\ybarnpone) ) \geq \|\ybarnpone - \ynpone\|^2 \geq 0. 
\end{equation}

Similarly, for $n\in I$, if we let $(x,y) = (x^n, \ybar^n)$ instead of $(x^*,y^*)$ in \eqref{eq:215x}, we have
\begin{equation}\label{eq:221x}
    2\rho(F(x^n,\ybar^n) - F(\xnpone,\ynpone)) \geq \|\ynpone - \ybar^n\|^2 \geq 0.
\end{equation}
For $n \in I_c$, $\ynpone = \ybar^n$;
from \eqref{eq:x_skip_opt}, since $\xnpone = \arg\min_x F(x,y)$ with $y=\ybar^n=\ynpone$,
\begin{equation}\label{eq:221xx}
    2\rho(F(x^n,\ybar^n) - F(\xnpone,\ynpone)) \geq 0.
\end{equation}
Hence, from  \eqref{eq:219x} and \eqref{eq:221x} to \eqref{eq:221xx},
$F(x^n,y^n) \geq F(x^n,\ybar^n) \geq F(\xnpone, \ynpone) \geq F(\xnpone,\ybarnpone)$.
Then, we have
\begin{equation}\label{eq:223x}
  \sum_{n=0}^{k-1}F(\xnpone,\ybarnpone) \geq k F(x^k,\ybar^k), \textrm{and} \quad \sum_{n\in I}F(\xnpone,\ynpone) \geq k_n F(x^k, y^k).
\end{equation}
Combining \eqref{eq:218x} and \eqref{eq:223x} yields
$2\rho (k+k_n)(F(x^*,y^*) - F(x^k,\ybar^k) ) \geq -\|\ybar^0 - y^*\|^2$.

\section{Derivation of the Stopping Criteria}\label{sec:stopping_criteria}
In this section, we show that the quantities that we use in our stopping criteria correspond to the primal and dual residuals \cite{boyd2010distributed} for the outer iterations and the gradient residuals for the inner iterations.  We first consider the inner iterations.
\begin{description}
  \item[FISTA-p] The necessary and sufficient optimality conditions for problem \eqref{eq:aug_lag_split} or \eqref{eq:aug_lag_split_simplified} are primal feasibility
\begin{equation}\label{eq:inner_partial_pfeas}
    \ybar^* - y^* = 0,
\end{equation}
and vanishing of the gradient of the objective function at $(x^*,\ybar^*)$, i.e.
\begin{eqnarray}
  0 &=& \nabla_x f(x^*,\ybar^*), \label{eq:inner_partial_xopt}\\
  0 &\in&  \nabla_y f(x^*,\ybar^*) + \partial g(\ybar^*). \label{eq:inner_partial_yopt}
\end{eqnarray}
Since $y^{k+1} = z^k$, the primal residual is thus $\ybar^{k+1} - y^{k+1} = \ybar^{k+1} - z^k$.
It follows from the optimality of $\xkpone$ in Line \ref{line:fistap_x} of Algorithm \ref{alg:fista_partial} that
\begin{eqnarray*}
  A^T(A\xkpone - b) - C^Tv^l + \frac{1}{\mu}C^T(C\xkpone - \ybarkpone) + \frac{1}{\mu}C^T(\ybarkpone - z^k) &=& 0 \\
  \nabla_x f(\xkpone,\ybarkpone) &=& \frac{1}{\mu}C^T(z^k - \ybarkpone).
\end{eqnarray*}
Similarly, from the optimality of $\ybarkpone$ in Line \ref{line:fistap_y}, we have that
\begin{eqnarray*}
  0 &\in& \partial g(\ybarkpone) + \nabla_y f(\xkpone,z^k) + \frac{1}{\rho}(\ybarkpone - z^k) \\
  &=& \partial g(\ybarkpone) + \nabla_y f(\xkpone,\ybarkpone) - \frac{1}{\mu}(\ybarkpone - z^k) + \frac{1}{\rho}(\ybarkpone - z^k) \\
  &=&  \partial g(\ybarkpone) + \nabla_y f(\xkpone,\ybarkpone),
\end{eqnarray*}
where the last step follows from $\mu = \rho$.  Hence, we see that $\frac{1}{\mu}C^T(z^k - \ybarkpone)$ is the gradient residual corresponding to \eqref{eq:inner_partial_xopt}, while \eqref{eq:inner_partial_yopt} is satisfied in every inner iteration.

  \item[APLM-S] The primal residual is $\ybarkpone - \ykpone$ from \eqref{eq:inner_partial_pfeas}.  Following the derivation for FISTA-p, it is not hard to verify that \eqref{eq:inner_partial_yopt} is always satisfied, and the gradient residual corresponding to \eqref{eq:inner_partial_xopt} is $\frac{1}{\mu}C^T(\ykpone - \ybarkpone)$.
%

  \item[FISTA] Similar to FISTA-p, the necessary and sufficient optimality conditions for problem \eqref{eq:aug_lag_fullsplit} are primal feasibility
\begin{equation*}
    (x^*,y^*) = (\xbar^*,\ybar^*),
\end{equation*}
and vanishing of the objective gradient at $(\xbar^*,\ybar^*)$,
\begin{eqnarray*}
  0 &=& \nabla_x f(\xbar^*,\ybar^*), \\
  0 &\in&  \nabla_y f(\xbar^*,\ybar^*) + \partial g(\ybar^*).
\end{eqnarray*}
Clearly, the primal residual is $(\xbarkpone-z_x^k, \ybarkpone - z_y^k)$ since $(\xkpone,\ykpone) \equiv (z_x^k, z_y^k)$.  From the optimality of $(\xbarkpone,\ybarkpone)$, it follows that
\begin{eqnarray*}
  0 &=& \nabla_x f(z_x^k,z_y^k) + \frac{1}{\rho}(\xbarkpone - z_x^k), \\
  0 &\in& \partial g(\ybarkpone) + \nabla_y f(z_x^k,z_y^k) + \frac{1}{\rho}(\ybarkpone - z_y^k).
\end{eqnarray*}
Here, we simply use $\frac{1}{\rho}(\xbarkpone - z_x^k)$ and $\frac{1}{\rho}(\ybarkpone - z_y^k)$ to approximate the gradient residuals.
\end{description}

Next, we consider the outer iterations.  The necessary and sufficient optimality conditions for problem \eqref{eq:overlap_glasso_aug1} are primal feasibility
\begin{equation}\label{eq:outer_primal_feas}
    Cx^* - y^* = 0,
\end{equation}
and dual feasibility
\begin{eqnarray}
  0 &=& \nabla L(x^*) - C^Tv^* \label{eq:outer_dual_feas1}\\
  0 &\in& \partial \tilde{\Omega}(y^*) + v^*. \label{eq:outer_dual_feas2}
\end{eqnarray}
Clearly, the primal residual is $r^l = Cx^l - y^l$.  The dual residual is $\left(
                                                                             \begin{array}{c}
                                                                               \nabla L(x^{l+1}) - C^T(v^l - \frac{1}{\mu}(Cx^{l+1}-\ybar^{l+1})) \\
                                                                               \partial\tilde{\Omega}(y^{l+1}) + v^l - \frac{1}{\mu}(Cx^{l+1}-\ybar^{l+1}) \\
                                                                             \end{array}
                                                                           \right)
$, recalling that $v^{l+1} = v^l - \frac{1}{\mu}(Cx^{l+1}-\ybar^{l+1})$.  The above is simply the gradient of the augmented Lagrangian \eqref{eq:aug_lag} evaluated at $(x^l,y^l,v^l)$.  Now, since the objective function of an inner iteration is the augmented Lagrangian with $v = v^l$, the dual residual for an outer iteration is readily available from the gradient residual computed for the last inner iteration of the outer iteration.

\section{Numerical Results}
\begin{table}
\begin{center}
\begin{tabular}{|c|c|c|c|c|c|}
\hline\hline
Data Sets & Algs & CPU (s) & Iters & Avg Sub-iters &$F(x)$\\
\hline
\multirow{4}{*}{ogl-5000-100-10-3}
&      ADAL& 	 1.70e+000& 	 61& 	 1.00e+000& 	 1.9482e+005\\
&    APLM-S& 	 1.71e+000& 	 8& 	 4.88e+000& 	 1.9482e+005\\
&   FISTA-p& 	 9.08e-001& 	 8& 	 4.38e+000& 	 1.9482e+005\\
&     FISTA& 	 2.74e+000& 	 10& 	 7.30e+000& 	 1.9482e+005\\
&  ProxGrad& 	 7.92e+001& 	 3858& 	 -& 	 -\\
\hline
\multirow{4}{*}{ogl-5000-600-10-3}
&      ADAL& 	 6.75e+001& 	 105& 	 1.00e+000& 	 1.4603e+006\\
&    APLM-S& 	 1.79e+002& 	 9& 	 1.74e+001& 	 1.4603e+006\\
&   FISTA-p& 	 4.77e+001& 	 9& 	 8.56e+000& 	 1.4603e+006\\
&     FISTA& 	 3.28e+001& 	 12& 	 1.36e+001& 	 1.4603e+006\\
&  ProxGrad& 	 7.96e+002& 	 5608& 	 -& 	 -\\
\hline
\multirow{4}{*}{ogl-5000-1000-10-3}
&      ADAL& 	 2.83e+002& 	 151& 	 1.00e+000& 	 2.6746e+006\\
&    APLM-S& 	 8.06e+002& 	 10& 	 2.76e+001& 	 2.6746e+006\\
&   FISTA-p& 	 2.49e+002& 	 10& 	 1.28e+001& 	 2.6746e+006\\
&     FISTA& 	 5.21e+001& 	 13& 	 1.55e+001& 	 2.6746e+006\\
&  ProxGrad& 	 1.64e+003& 	 6471& 	 -& 	 -\\
\hline\hline
\end{tabular}
\caption{Numerical results for ogl set 1.  For ProxGrad, Avg Sub-Iters and $F(x)$ fields are not applicable since the algorithm is not based on an outer-inner iteration scheme, and the objective function that it minimizes is different from ours.  We tested ten problems with $J = 100, \cdots, 1000$, but only show the results for three of them to save space.}
\label{tab:chen_data1}
\end{center}
\end{table}

\begin{table}
\begin{center}
\begin{tabular}{|c|c|c|c|c|c|}
\hline\hline
Data Sets & Algs & CPU (s) & Iters & Avg Sub-iters &$F(x)$\\
\hline
\multirow{5}{*}{ogl-1000-200-10-3}
&      ADAL& 	 4.18e+000& 	 77& 	 1.00e+000& 	 9.6155e+004\\
&    APLM-S& 	 1.64e+001& 	 9& 	 2.32e+001& 	 9.6156e+004\\
&   FISTA-p& 	 3.85e+000& 	 9& 	 1.02e+001& 	 9.6156e+004\\
&     FISTA& 	 2.92e+000& 	 11& 	 1.44e+001& 	 9.6158e+004\\
&  ProxGrad& 	 1.16e+002& 	 4137& 	 -& 	 -\\
\hline
\multirow{5}{*}{ogl-5000-200-10-3}
&      ADAL& 	 5.04e+000& 	 63& 	 1.00e+000& 	 4.1573e+005\\
&    APLM-S& 	 8.42e+000& 	 8& 	 8.38e+000& 	 4.1576e+005\\
&   FISTA-p& 	 3.96e+000& 	 9& 	 6.56e+000& 	 4.1572e+005\\
&     FISTA& 	 6.54e+000& 	 10& 	 9.70e+000& 	 4.1573e+005\\
&  ProxGrad& 	 1.68e+002& 	 4345& 	 -& 	 -\\
\hline
\multirow{5}{*}{ogl-10000-200-10-3}
&      ADAL& 	 6.41e+000& 	 44& 	 1.00e+000& 	 1.0026e+006\\
&    APLM-S& 	 1.46e+001& 	 10& 	 7.60e+000& 	 1.0026e+006\\
&   FISTA-p& 	 5.60e+000& 	 10& 	 5.50e+000& 	 1.0026e+006\\
&     FISTA& 	 1.09e+001& 	 10& 	 8.50e+000& 	 1.0027e+006\\
&  ProxGrad& 	 3.31e+002& 	 6186& 	 -& 	 -\\
\hline\hline
\end{tabular}
\caption{Numerical results for ogl set 2.  We ran the test for ten problems with $n = 1000, \cdots, 10000$, but only show the results for three of them to save space.}
\label{tab:chen_data2}
\end{center}
\end{table}


\begin{table}
\begin{center}
\begin{tabular}{|c|c|c|c|c|c|}
\hline\hline
Data Sets & Algs & CPU (s) & Iters & Avg Sub-iters &$F(x)$\\
\hline
\multirow{3}{*}{ogl-dct-1000-5000-1}
&	ADAL&	1.14e+001&	194&	1.00e+000&	8.4892e+002\\
&	FISTA-p&	1.21e+001&	20&	1.11e+001&	8.4892e+002\\
&	FISTA&	2.49e+001&	24&	2.51e+001&	8.4893e+002\\
\hline
\multirow{3}{*}{ogl-dct-1000-10000-1}
&	ADAL&	3.31e+001&	398&	1.00e+000&	1.4887e+003\\
&	FISTA-p&	2.54e+001&	41&	5.61e+000&	1.4887e+003\\
&	FISTA&	6.33e+001&	44&	1.74e+001&	1.4887e+003\\
\hline
\multirow{3}{*}{ogl-dct-1000-15000-1}
&	ADAL&	6.09e+001&	515&	1.00e+000&	2.7506e+003\\
&	FISTA-p&	3.95e+001&	52&	4.44e+000&	2.7506e+003\\
&	FISTA&	9.73e+001&	54&	1.32e+001&	2.7506e+003\\
\hline
\multirow{3}{*}{ogl-dct-1000-20000-1}
&	ADAL&	9.52e+001&	626&	1.00e+000&	3.3415e+003\\
&	FISTA-p&	6.66e+001&	63&	6.10e+000&	3.3415e+003\\
&	FISTA&	1.81e+002&	64&	1.61e+001&	3.3415e+003\\
\hline
\multirow{3}{*}{ogl-dct-1000-25000-1}
&	ADAL&	1.54e+002&	882&	1.00e+000&	4.1987e+003\\
&	FISTA-p&	7.50e+001&	88&	3.20e+000&	4.1987e+003\\
&	FISTA&	1.76e+002&	89&	8.64e+000&	4.1987e+003\\
\hline
\multirow{3}{*}{ogl-dct-1000-30000-1}
&	ADAL&	1.87e+002&	957&	1.00e+000&	4.6111e+003\\
&	FISTA-p&	8.79e+001&	96&	2.86e+000&	4.6111e+003\\
&	FISTA&	2.24e+002&	94&	8.54e+000&	4.6111e+003\\
\hline\hline
\end{tabular}
\caption{Numerical results for dct set 2 (scalability test) with $l_1/l_2$-regularization.  All three algorithms were ran in factorization mode with a fixed $\mu = \mu_0$.}
\label{tab:dct_scal_results_l12}
\end{center}
\end{table}

\begin{table}
\begin{center}
\begin{tabular}{|c|c|c|c|c|c|}
\hline\hline
Data Sets & Algs & CPU (s) & Iters & Avg Sub-iters &$F(x)$\\
\hline
\multirow{4}{*}{ogl-dct-1000-5000-1}
&	ADAL&	1.53e+001&	266&	1.00e+000&	7.3218e+002\\
&	FISTA-p&	1.61e+001&	10&	3.05e+001&	7.3219e+002\\
&	FISTA&	3.02e+001&	16&	4.09e+001&	7.3233e+002\\
&	ProxFlow&	1.97e+001&	-&	-&	7.3236e+002\\
\hline
\multirow{4}{*}{ogl-dct-1000-10000-1}
&	ADAL&	3.30e+001&	330&	1.00e+000&	1.2707e+003\\
&	FISTA-p&	3.16e+001&	10&	3.10e+001&	1.2708e+003\\
&	FISTA&	7.27e+001&	24&	3.25e+001&	1.2708e+003\\
&	ProxFlow&	3.67e+001&	-&	-&	1.2709e+003\\
\hline
\multirow{4}{*}{ogl-dct-1000-15000-1}
&	ADAL&	4.83e+001&	328&	1.00e+000&	2.2444e+003\\
&	FISTA-p&	5.40e+001&	15&	2.52e+001&	2.2444e+003\\
&	FISTA&	8.64e+001&	23&	2.66e+001&	2.2449e+003\\
&	ProxFlow&	9.91e+001&	-&	-&	2.2467e+003\\
\hline
\multirow{4}{*}{ogl-dct-1000-20000-1}
&	ADAL&	8.09e+001&	463&	1.00e+000&	2.6340e+003\\
&	FISTA-p&	8.09e+001&	16&	2.88e+001&	2.6340e+003\\
&	FISTA&	1.48e+002&	26&	2.93e+001&	2.6342e+003\\
&	ProxFlow&	2.55e+002&	-&   -&	2.6357e+003\\
\hline
\multirow{4}{*}{ogl-dct-1000-25000-1}
&	ADAL&	7.48e+001&	309&	1.00e+000&	3.5566e+003\\
&	FISTA-p&	1.15e+002&	30&	1.83e+001&	3.5566e+003\\
&	FISTA&	2.09e+002&	38&	2.30e+001&	3.5568e+003\\
&	ProxFlow&	1.38e+002&	-&	-&	3.5571e+003\\
\hline
\multirow{4}{*}{ogl-dct-1000-30000-1}
&	ADAL&	9.99e+001&	359&	1.00e+000&	3.7057e+003\\
&	FISTA-p&	1.55e+002&	29&	2.17e+001&	3.7057e+003\\
&	FISTA&	2.60e+002&	39&	2.25e+001&	3.7060e+003\\
&	ProxFlow&	1.07e+002&	-&	-&	3.7063e+003\\
\hline\hline
\end{tabular}
\caption{Numerical results for dct set 2 (scalability test) with $l_1/l_\infty$-regularization.  The algorithm configurations are exactly the same as in Table \ref{tab:dct_scal_results_l12}.}
\label{tab:dct_scal_results_l1inf}
\end{center}
\end{table}

\begin{table}
\begin{center}
\begin{tabular}{|c|c|c|c|c|c|}
\hline\hline
Data Sets & Algs & CPU (s) & Iters & Avg Sub-iters &$F(x)$\\
\hline
\multirow{3}{*}{ogl-dct-1000-5000-1}
&	FISTA-p&	1.83e+001&	12&	2.34e+001&	8.4892e+002\\
&	FISTA&	2.49e+001&	24&	2.51e+001&	8.4893e+002\\
&	ADAL&	1.35e+001&	181&	1.00e+000&	8.4892e+002\\
\hline
\multirow{3}{*}{ogl-dct-1000-10000-1}
&	FISTA-p&	3.16e+001&	14&	1.73e+001&	1.4887e+003\\
&	FISTA&	6.33e+001&	44&	1.74e+001&	1.4887e+003\\
&	ADAL&	4.43e+001&	270&	1.00e+000&	1.4887e+003\\
\hline
\multirow{3}{*}{ogl-dct-1000-15000-1}
&	FISTA-p&	4.29e+001&	14&	1.51e+001&	2.7506e+003\\
&	FISTA&	9.73e+001&	54&	1.32e+001&	2.7506e+003\\
&	ADAL&	5.37e+001&	216&	1.00e+000&	2.7506e+003\\
\hline
\multirow{3}{*}{ogl-dct-1000-20000-1}
&	FISTA-p&	7.53e+001&	13&	2.06e+001&	3.3416e+003\\
&	FISTA&	1.81e+002&	64&	1.61e+001&	3.3415e+003\\
&	ADAL&	1.57e+002&	390&	1.00e+000&	3.3415e+003\\
\hline
\multirow{3}{*}{ogl-dct-1000-25000-1}
&	FISTA-p&	7.41e+001&	15&	1.47e+001&	4.1987e+003\\
&	FISTA&	1.76e+002&	89&	8.64e+000&	4.1987e+003\\
&	ADAL&	8.79e+001&	231&	1.00e+000&	4.1987e+003\\
\hline
\multirow{3}{*}{ogl-dct-1000-30000-1}
&	FISTA-p&	8.95e+001&	14&	1.58e+001&	4.6111e+003\\
&	FISTA&	2.24e+002&	94&	8.54e+000&	4.6111e+003\\
&	ADAL&	1.12e+002&	249&	1.00e+000&	4.6111e+003\\
\hline\hline
\end{tabular}
\caption{Numerical results for the DCT set with $l_1/l_2$-regularization.  FISTA-p and ADAL were ran in PCG mode with the dynamic scheme for updating $\mu$.  $\mu$ was fixed at $\mu_0$ for FISTA.}
\label{tab:tony_pcg_l12}
\end{center}
\end{table}

\begin{table}
\begin{center}
\begin{tabular}{|c|c|c|c|c|c|}
\hline\hline
Data Sets & Algs & CPU (s) & Iters & Avg Sub-iters &$F(x)$\\
\hline
\multirow{2}{*}{ogl-dct-1000-5000-1}
&	FISTA-p&	2.30e+001&	11&	2.93e+001&	7.3219e+002\\
&	ADAL&	1.89e+001&	265&	1.00e+000&	7.3218e+002\\
&	FISTA&	3.02e+001&	16&	4.09e+001&	7.3233e+002\\
&	ProxFlow&	1.97e+001&	-&	-&	7.3236e+002\\
\hline
\multirow{2}{*}{ogl-dct-1000-10000-1}
&	FISTA-p&	5.09e+001&	11&	3.16e+001&	1.2708e+003\\
&	ADAL&	4.77e+001&	323&	1.00e+000&	1.2708e+003\\
&	FISTA&	7.27e+001&	24&	3.25e+001&	1.2708e+003\\
&	ProxFlow&	3.67e+001&	-&	-&	1.2709e+003\\
\hline
\multirow{2}{*}{ogl-dct-1000-15000-1}
&	FISTA-p&	6.33e+001&	12&	2.48e+001&	2.2445e+003\\
&	ADAL&	9.41e+001&	333&	1.00e+000&	2.2444e+003\\
&	FISTA&	8.64e+001&	23&	2.66e+001&	2.2449e+003\\
&	ProxFlow&	9.91e+001&	-&	-&	2.2467e+003\\
\hline
\multirow{2}{*}{ogl-dct-1000-20000-1}
&	FISTA-p&	8.21e+001&	12&	2.42e+001&	2.6341e+003\\
&	ADAL&	1.59e+002&	415&	1.00e+000&	2.6340e+003\\
&	FISTA&	1.48e+002&	26&	2.93e+001&	2.6342e+003\\
&	ProxFlow&	2.55e+002&	-&	-&	2.6357e+003\\
\hline
\multirow{2}{*}{ogl-dct-1000-25000-1}
&	FISTA-p&	1.43e+002&	13&	2.98e+001&	3.5567e+003\\
&	ADAL&	1.20e+002&	310&	1.00e+000&	3.5566e+003\\
&	FISTA&	2.09e+002&	38&	2.30e+001&	3.5568e+003\\
&	ProxFlow&	1.38e+002&	-&	-&	3.5571e+003\\
\hline
\multirow{2}{*}{ogl-dct-1000-30000-1}
&	FISTA-p&	1.75e+002&	13&	3.18e+001&	3.7057e+003\\
&	ADAL&	2.01e+002&	361&	1.00e+000&	3.7057e+003\\
&	FISTA&	2.60e+002&	39&	2.25e+001&	3.7060e+003\\
&	ProxFlow&	1.07e+002&	-&	-&	3.7063e+003\\
\hline\hline
\end{tabular}
\caption{Numerical results for the DCT set with $l_1/l_\infty$-regularization.  FISTA-p and ADAL were ran in PCG mode. The dynamic updating scheme for $\mu$ was applied to FISTA-p, while $\mu$ was fixed at $\mu_0$ for ADAL and FISTA.  }
\label{tab:tony_pcg_l1inf}
\end{center}
\end{table}

\begin{table}
\begin{center}
\begin{tabular}{|c|c|c|c|c|c|}
\hline\hline
Data Sets & Algs & CPU (s) & Iters & Avg Sub-iters &$F(x)$\\
\hline
\multirow{6}{*}{BreastCancerData}
&      ADAL& 	 6.24e+000& 	 136& 	 1.00e+000& 	 2.9331e+003\\
&    APLM-S& 	 4.02e+001& 	 12& 	 4.55e+001& 	 2.9331e+003\\
&   FISTA-p& 	 6.86e+000& 	 12& 	 1.48e+001& 	 2.9331e+003\\
&     FISTA& 	 5.11e+001& 	 75& 	 1.29e+001& 	 2.9340e+003\\
&  ProxGrad& 	 7.76e+002& 	 6605& 	 1.00e+000& 	 -\\
\hline\hline
\end{tabular}
\caption{Numerical results for Breast Cancer Data using $l_1/l_2$-regularization.  In this experiment, we kept $\mu$ constant at 0.01 for ADAL.  The CPU time is for a single run on the entire data set with the value of $\lambda$ selected to minimize the RMSE in Figure \ref{fig:breast_cancer_plots}. }
\label{tab:breast_cancer_results}
\end{center}
\end{table}

\bibliographystyle{abbrv}
\bibliography{tony_bib}

\end{document}